\documentclass[11pt,a4paper]{amsart}
\usepackage{amsfonts,amsthm,amsmath,amssymb}
\usepackage{amscd}
\usepackage[T1]{fontenc}
\usepackage[mathscr]{eucal}
\usepackage{graphicx}
\usepackage{graphics}
\usepackage{subcaption}
\usepackage[hidelinks]{hyperref}
\usepackage{multicol}
\usepackage{pict2e}
\usepackage{epic}
\usepackage{mhequ}
\usepackage[ruled]{algorithm2e}
\usepackage[numbers,sort&compress]{natbib}

\numberwithin{equation}{section}

\usepackage{comment}
\usepackage{microtype}
\usepackage{fullpage}
\usepackage{aamath}
\usepackage{aageneral}
\usepackage{todonotes}
\usepackage{microtype}

\usepackage{thmtools,thm-restate}

\usepackage{bm}

\setuptodonotes{color=orange!30}

 \theoremstyle{plain}
 \newtheorem{Thm}{Theorem}[section]
\newtheorem{Lemma}[Thm]{Lemma}
\newtheorem{Cor}[Thm]{Corollary}

\theoremstyle{definition}

\newtheorem{Rem}[Thm]{Remark}

\newcommand{\aanote}[1]{\todo[color=red!35]{\small AA: #1}}

\newcommand{\ovl}{\overline}

\newcommand{\Rr}{\mathbb{R}}
\newcommand{\dtheta}{\Delta \theta}

\newcommand*{\transpose}{^{\mkern-1.5mu\mathsf{T}}}

\DeclarePairedDelimiter\norm{\lVert}{\rVert}
\DeclarePairedDelimiterX\set[1]\lbrace\rbrace{\def\given{\;\delimsize\vert\;}#1}

\DeclareMathOperator{\Tr}{Tr}

\def\mathunderline#1#2{\color{#1}\underline{{\color{black}#2}}\color{black}}

\begin{document}
\onecolumn
\title{A Homotopic Approach to Policy Gradients\\  for Linear Quadratic Regulators with Nonlinear Controls}

\author{Craig Xu Chen}

\author{Andrea Agazzi}

\address{Department of Mathematics\\Duke University\\Durham, NC 27708 }
\email{craig.chen@duke.edu \and agazzi@math.duke.edu}

\begin{abstract}
We study the convergence of deterministic policy gradient algorithms in continuous state and action space for the prototypical Linear Quadratic Regulator (LQR) problem when the search space is not limited to the family of linear policies. We first provide a counterexample showing that extending the policy class to piecewise linear functions results in local minima of the policy gradient algorithm. To solve this problem, we develop a new approach that involves sequentially increasing a discount factor between iterations of the original policy gradient algorithm. We finally prove that this homotopic variant of policy gradient methods converges to the global optimum of the undiscounted Linear Quadratic Regulator problem for a large class of Lipschitz, non-linear policies.
\end{abstract}

\maketitle

\section{Introduction}

Recent advances in machine learning have strongly depended on techniques from reinforcement learning to generate their superhuman results on tasks that were previously deemed ``unachievable'' by artificial intelligence models. Some of the most spectacular examples are the recent breakthroughs in the games of Go and Starcraft 2 \cite{alphago, alphazero, alphastar} and the robotics tasks such as learning how to walk or controlling a hand \cite{RLwalk, robothand}.

A shared aspect of these recent advances is the use of \emph{policy gradient} algorithms (and their variants) \cite{sutton2000}. Given the simplicity of these algorithms, their success in solving non-convex optimization problems is rather surprising. In the past few years, there has been significant progress made in the development of our understanding of the convergence of these algorithms - especially in the tabular setting \cite{haarnoja18, mei2020}. However, many of the games and control problems mentioned in the first paragraph have incredibly large (or infinite in the case of robotics) state and action spaces, so that tabular representation is computationally prohibitive. Due to this complexity, none of the above breakthroughs would have been possible without the use of \emph{nonlinear function approximation} -- most commonly, (deep) neural networks -- in combination with policy gradient algorithms: With parametric function approximation, only the parameters of the approximator need to be stored rather than a table of all possible game-states. 

Nevertheless, the current body of theoretical work still does not explain the success policy gradient methods have experienced in Go and Starcraft 2 due to an insufficient understanding of how such methods behave when neural networks are used as the function approximator of choice. One obstruction in this sense stems from the nonlinear behavior of neural networks, resulting in optimization landscapes that are highly non-convex \cite{goodfellow2015}, preventing us from using many of the tools developed in classical optimization theory. This large gap between empirical success and theory of nonlinear function approximation in reinforcement learning algorithms is also related to our limited theoretical understanding of their interplay. In this paper, we contribute to filling this gap by proving convergence results for nonlinear deterministic policies in a prototypical control problem with continuous state and action space: the Linear Quadratic Regulator (LQR) problem. In previous work by \cite{bhandari19, fazel2018, bu2019, yang2019}, global convergence results have been obtained in this context by restricting to the linear policy class. This choice of search space favorably aligns with the policy gradient update since the optimal policy of the LQR is known to be a linear function of the state. Our work extends these results to a larger class of non-linear policies. To do so, we first construct an example to demonstrate the potential for negative interplay between the expressiveness of the policy class and the local nature of the policy gradient algorithm. We then consider a new perspective in solving reinforcement learning problems that involves iteratively updating a discount factor. With this perspective, we reduce the global convergence problem to a local one, allowing us to work with a wider policy class and to avoid potential local minima. These stronger convergence results do not come without a cost, though; the process of iteratively updating a discount factor necessarily results in more training iterations.

\subsection{Related Works}~\\

In light of the advantages provided by policy-based approaches to reinforcement learning, there has been a recent wave of research seeking to provide theoretical guarantees for these algorithms that have been shown to work in practice.

Most of the recent work on the convergence of policy gradient methods focuses on the finite action space setting. In this context, \cite{liu2019} has shown that, in finite action spaces, algorithms like TRPO \cite{schulman2015} and PPO \cite{schulman2018} converge to the globally optimal policy. More results on the global convergence of policy gradient dynamics and extensions to function approximation are obtained in \cite{agarwal19, mei2020, wang2020, li2021}.  However, in the aforementioned publications, the assumption of a finite action space is essential in obtaining the results; when considering continuous action spaces, we use the LQR as a proxy for more general environments since the LQR setting is well understood and allows for more straightforward analysis.

In the continuous action-space setting, \cite{fazel2018} show that deterministic policy gradient methods for the policy class of \emph{linear policies} converge to the globally optimal policy of the LQR problem. Building off of \cite{fazel2018}, \cite{bu2019} provides strong positive results in an explicit characterization of the LQR when learning with linear policies showing that the cost is smooth and coercive and that the discretization of various gradient flows lead to algorithms with linear convergence rates. Similarly, \cite{yang2019} provides a convergence guarantee for Actor-Critic methods in the LQR setting by considering the class of linear-Gaussian policies with fixed variance. Other recent work \cite{hambly2020} also considers the class of linear policies but works in the \emph{finite time-horizon} setting. Rather than working with a more complex learning task such as the finite-horizon LQR, our work instead tackles the problem of nonlinear function approximation and thus we prove our results in the simpler setting of the infinite-horizon, noiseless LQR; however, unlike previous work in this setting, we consider a wider (deterministic) policy class that can be understood as linear combinations of nonlinear functions. Furthermore, most works only consider the evolution of the policy with respect to time/iterations; we also consider the evolution of the optimal policy with respect to the discount factor. This new perspective allows us to use some new tools in proving the convergence of policy gradient methods in the LQR.

{Most relevant to our work, \cite{lavet2016} discusses the effect of changing the dscount factor during training of reinforcement learning algorithms. This study, although strictly empirical, shows that our homotopic variant of policy gradient methods can reduce the number of training iterations needed to learn an optimal policy, \ie that this modified reinforcement learning procedure not only improves the optimality properties of reinforcement learning algorithms but also their convergence speed. On the theoretical side, a team at DeepMind \cite{tan21} has studied how the \emph{value} function behaves as function of the discount factor and develop a set of interpolating objectives to go from a myopic value function $V_\gamma$ to the true objective $V_{\gamma^*}$; they also show that their family of interpolating objectives yields faster convergence and better asymptotic performance on MuJoCo tasks.}

\subsection{Contributions}~\\

Our work continues along a line of work using the LQR as a proxy for more general learning environments \cite{bradtke1993, hazan2017, hazan2018, arora2018, tu2018, dean2018, fazel2018, gravell2020, hambly2020}. In this paper, we present a homotopy-based approach inspired by the human learning process to iteratively update the policy until reaching the globally optimal policy, and in particular, we directly extend the result of \cite{fazel2018} to nonlinear policy classes. We show that by continuously updating the discount factor after policy improvements, one can reach the global optimum regardless of the quality of the initial policy, subject to a few assumptions on the parametrization and expressiveness of the policy (\eg our policy parametrization is linear in the parameters). Although this approach can be understood as a new type of policy-gradient algorithm, our focus is theoretical and we mainly consider the homotopy approach as a tool to prove convergence results.

The intuition behind the homotopy-based algorithm comes from the observation that when teaching children, one does not expect the child to consider the long term consequences of their actions. Instead, children are first taught to learn the basics (i.e., learn with a discount factor close to zero) before moving on to developing their long-term planning ability (increasing to a larger discount factor). Heuristically, one can make sense of this by reasoning that it’s likely easier to learn the optimal policy having learned the basics than it is learning it from scratch. Mathematically, the homotopy approach relies on the fact that the optimal policy of the LQR is continuous with respect to the discount factor. In more challenging environments such as Chess or Go, this may not be the case.

Summarizing, the main contributions of this work are as follows:

\begin{itemize}
\item We show that one can not always expect policy gradient methods to converge even with relatively well-behaved function approximators. We provide an example policy that is linear in the parameters, contains the globally optimal policy, but does \emph{not} converge to such globally optimal policy when trained with a fixed discount factor. Furthermore, this example satisfies all assumptions needed for our convergence result.

\item We formulate the policy gradient learning process as a homotopy between various optimal policies along the discount factor.
 By leveraging this idea, we provide a convergence result for an expanded policy class consisting of linear combinations of non-linear functions (in state space) under some mild assumptions on the expressiveness and parametrization of the policy class (see Assumptions \ref{asm:linear indep}, \ref{asm:linear functions in span}).
\end{itemize}

\noindent As a corollary, our result also shows that the assumption of a stabilizing initial policy in \cite{fazel2018} is not necessary. By using our homotopy algorithm, \emph{any} linear policy will converge to the optimal policy of the undiscounted LQR. This is ultimately a trade-off between initial information and number of training iterations; the homotopy approach allows us to start the learning process with less background information at the cost of more training iterations. This type of trade-off has been considered before, most notably in \cite{alphago} and \cite{alphazero} where AlphaGo was initially trained on human games via supervised learning before engaging in self-play but AlphaZero was trained entirely via self-play and no human input.

\subsection{Organization}~\\

The paper is organized as follows. In the coming section, we provide the necessary background information to understand the context of our problem (\sref{sec:prelims}). In \sref{sec:homotopy} we provide our result regarding the convergence of policy gradient methods with a non-linear function approximator in the LQR using the aforementioned homotopy approach inspired by the human learning process. In this section, we also provide the counterexample where policy gradient methods do not converge to the global optimum. The discussion is contained in \sref{sec:Conclusion}. The full proofs of the results, as well as some numerical examples corroborating our claims, obtained are housed in Appendices \ref{apdx:figures}, \ref{apdx:counterex proofs} and \ref{apdx:homotopy proofs}.

\section{Preliminaries and Background}\label{sec:prelims}

In this paper, we study a reinforcement learning problem where an agent interacts in discrete time with the environment with the goal of minimizing some cost function that depends on the agent's states and actions.  The agent does this by sequentially choosing actions over time which influence the agent's state in the environment. Throughout, for given $n,m\in \mathbb N$ we respectively define $X = \mathbb R^n$ and $U = \mathbb R^m$ as the set of possible states and actions of the agent. Correspondingly, for $t \in \mathbb Z_{\geq 0}$ we denote by $x_t \in X$ and $u_t \in U$ the state and action of the agent at time $t$. The function that maps a state to action of the agent in that state is called the \emph{policy} $\pi~:~X \to U$, and the motivation behind this paper is to investigate the conditions under which an agent will learn the optimal policy of a specific environment: the \emph{Linear-Quadratic Regulator}.

\subsection{The Linear-Quadratic Regulator}~\\

The Linear-Quadratic Regulator (LQR) is a classic optimal control problem.
In this paper, we are interested in the special case where the dynamics are \emph{linear}, time-invariant with no disturbance or added noise and the cost function is \emph{quadratic} in the state and the control action. We consider the discounted infinite time-horizon problem,
\begin{align}
\text{minimize}
&\quad
	\Ex{x_0 \sim \rho_0}{ \sum_{t=0}^{\infty} \gamma^t
	\pc {x_t^\intercal Qx_t + u_t^\intercal Ru_t} }\label{eq:cost}\\
\text{with}
&\quad x_{t+1} = Ax_{t} + Bu_{t},  \stepcounter{equation}\tag{\theequation}\label{eq:lqr dynamics}
\end{align}
where $A~:~X \to X$ and $B~:~U \to X$ are the system (or transition) matrices, $Q~:~X \to X$ and $R~:~U \to U$ are positive definite cost matrices, $x_0 \in X$ is randomly distributed according to  probability distribution $\rho_0$ on $X$, and $\gamma \in [0,1]$ is the discount factor. For a given policy $\pi$, the expression to be minimized in \eref{eq:cost} is referred to as the \emph{cost} associated to $\pi$, which we denote by $C(\pi)$. In the entirety of this work, the pair $(A,B)$ is assumed to be controllable and $(A,D)$ is assumed to be observable (where $Q = D\transpose D$).

It is known \cite{AndersonMooreControlTheory}  that the optimal control for the LQR problem is a linear function of the state,
\begin{equation*}
u_t = K^* x_t.
\end{equation*}
If $A,B,Q,R,\gamma$ are known, the optimal control $K^*$ can be calculated explicitly. Let $P_\gamma$ denote the unique positive definite solution to the discounted discrete-time algebraic Riccati equation (DARE)
\begin{equation*}
P_\gamma
=
	\gamma A^\intercal P_\gamma A - \gamma^2 A^\intercal P_\gamma B
	\pc{
		R + \gamma B^\intercal P_\gamma B
	}^{-1}
	B^\intercal P_\gamma A + Q\,,
\end{equation*}
for which we know the solution exists and since $(A,B)$ is assumed controllable and $(A,D)$ is assumed observable \cite{AndersonMooreControlTheory}.
We can then write the optimal control as
\begin{equation}\label{eq:lqr optimal control}
K_\gamma^*
=
-\gamma \pc{2
	R + \gamma B^\intercal P_\gamma B
}^{-1} B^\intercal P_\gamma A.
\end{equation}
Later, we will sometimes omit the subscript $\gamma$ when the dependencies are clear from context.

Note that, in general, the optimal policy for the discounted LQR is different from the optimal policy for the undiscounted LQR. For instance, the optimal policy with respect to the discounted LQR may not even be stabilizing \cite{postoyan2017}.

Throughout the remainder of the paper, we will use $\gamma$ to denote the discount rate which we choose given an initial policy. We assume that $\gamma \in [0,1]$ is small enough such that the cost function is still bounded from above for all initial states $x_0 \in \supp(\rho_0)$, where $\rho_0$ is the distribution of initial states.

\subsubsection*{Why LQR?} Although the assumptions for the linear-quadratic system are rather restrictive, the LQR is still an important problem for multiple reasons. Firstly, it is one of the few problems where the closed-form of the optimal control exists. This fact by itself makes the LQR a pleasant environment to work with. Secondly, the LQR has a continuous state and action space. Thirdly, despite its simplicity, there is still a lack of theoretical guarantees regarding convergence of reinforcement learning algorithms in this setting. All of these factors have certainly contributed to its re-emergence as a benchmark environment in reinforcement learning theory; it is both nice to work with mathematically and difficult enough to be applicable to the real world.

\subsection{Policy Gradient Methods}~\\

Policy gradient methods are a general class of reinforcement learning algorithms where the agent directly learns a policy -- which is assumed to be in a certain parametric family of functions -- rather than first learning a value function. During training, the parameters of the policy are updated via gradient descent on the cost function. As mentioned in the introduction, policy gradient methods have increasingly become  the approach of choice for solving difficult reinforcement learning problems. One reason for this popularity is that policy-based methods, as opposed to value-based, can easily learn stochastic policies, whereas action-value based methods have no natural \emph{and} flexible way to do so. 

The fundamental result underlying the successful application of these methods is the policy gradient theorem \cite{sutton2000} which provides an explicit formula for the gradient of the performance in terms of the gradient of the policy w.r.t its parameters -- importantly, the gradient of the state distribution is \emph{not} needed, greatly simplifying the application of these methods.


The analogous theorem in the deterministic setting (first proven by \cite{silver14}) allows one to conveniently compute the gradient of the cost function when considering a class of deterministic policies, as is the case in this paper. 
Furthermore, as shown in \cite{fazel2018}, it is possible to explicitly compute the gradient of the LQR cost function for linear policies. They show that for a linear policy $K$,
\begin{equ}
\nabla C(K)
=
2 \pc{ (R+ \gamma B\transpose P_K B) K + \gamma B\transpose P_K A} \Sigma_K
\end{equ}
where $P_K$ denotes the ``cost to-go'' matrix (\ie $C(K) = \Ex{x_0 \sim \rho_0}{x_0\transpose P_K x_0}$) and $\Sigma_K$ is the state covariance matrix.

In our work, we consider the same LQR setup as in \cite{fazel2018}. That is, we also consider a deterministic policy rather than a stochastic policy, a randomly distributed initial state, and noiseless dynamics. Note that it is known \cite{AndersonMooreControlTheory} that the inclusion of additive zero-mean white noise to the LQR dynamics \eref{eq:lqr dynamics} does not change the optimal control.

\section{Main Results} \label{sec:homotopy}

\subsubsection*{Notation.} We use $\norm{ \cdot }$ as the operator norm of a matrix or the Euclidean norm of a vector. We will use $\lambda_{\mathrm{min}}(\cdot)$ to refer to the smallest eigenvalue of a matrix.

In the first subsection, we construct the aforementioned example where vanilla (fixed discount factor) policy gradients only yield convergence to a \emph{local} optimum. Importantly, this policy class of the example satisfies the assumptions needed for our convergence result. In the following subsection, we provide the main theorems which show convergence of the homotopy algorithm. The lemmas and additional details for both subsections have been omitted for clarity, the full details can be found in Appendices \ref{apdx:homotopy proofs} and \ref{apdx:counterex proofs}.
{We denote by $\set{x_t}$ the sequence generated by the dynamics $x_{t+1} = Ax_t + B\pi_{\theta}(x_t)$. Throughout, $\pi_{\theta}$ denotes the parametric policy of choice evaluated with parameters $\theta \in \mathbb R^d$.}
In this paper, we consider policies of the form
\begin{equ}\label{eq:policy class}
\pi_{\theta}(x_t) = \sum_{k=1}^{d} \theta_k f_k(x_t)
\end{equ}
where each $f_k: \Rr^n \to \Rr^m$ is a non-linear, globally Lipschitz function and $\theta_k \in \Rr$ is a parameter to be learned.

To establish our results we make the following assumptions:
\begin{restatable}{Asm}{AsmOne}\label{asm:linear indep}
The functions $f_k$ are linearly independent.
\end{restatable}
\begin{restatable}{Asm}{AsmTwo}\label{asm:linear functions in span}
For the LQR problem defined by ($A,B,Q,R$), for all $\gamma \in [0,1]$, the $\gamma$-optimal policy $K_\gamma^*$ can be represented by $\pi_\theta$, \emph{i.e.} $ K_\gamma^*\in \text{Span}\set{f_k}$.
\end{restatable}

Assumption \ref{asm:linear indep} is a standard assumption when considering function approximation via linear combination. We need this assumption to guarantee the uniqueness of policy parametrization. Assumption \ref{asm:linear functions in span} is necessary to guarantee that, for any choice of $\gamma$,  the optimal policy is in the parametric policy class being considered.

{We note that the assumption of linear independence is, however, not essential for our proofs; rather, it is made for technical convenience to simplify our proofs of Lemma \ref{lem:secondorderapprox} and Theorem \ref{thm:convergence of first iteration}, guaranteeing local strict convexity of the landscape. For example, in the proof of Lemma \ref{lem:secondorderapprox}, if we were to relax Assumption \ref{asm:linear indep}, we would see that the eigenvectors that now correspond to the zero eigenvalue of the Hessian would be the directions where the policy is unaffected by changing the parameters. Therefore, we could relax the assumption to allow overparameterization and in turn to allow for wider networks in Remark \ref{rem:random features model}.}

\begin{restatable}{Lemma}{LemHomotopy}\label{lem:homotopy}
Assumptions \ref{asm:linear indep} and \ref{asm:linear functions in span} imply that the optimal parameters are continuous in the discount factor. In other words, for $\gamma \in [0,1]$ and for any $\epsilon > 0$, $\exists \delta > 0$ such that for $\gamma' \in [0,1]$ such that $\abs{\gamma - \gamma'} < \delta$, $\norm{\theta_{\gamma}^* - \theta_{\gamma'}^*} < \epsilon$, where $\theta_{\gamma}^*$ denotes the parameters that correspond to $K_{\gamma}^*$ for the policy class of \eref{eq:policy class}.
\end{restatable}

\begin{proof}
Let $\set{g_k}$ denote the orthonormal set obtained from $\set{f_k}$ via Gram-Schmidt. We can then write
\begin{equ}
\pi_{\theta}
=
	\sum_{k=1}^{d} \theta_k \sum_{l=1}^{d} \dtp{f_k,g_l} g_l
=
	\sum_{l=1}^{d} \dtp*{\sum_{k=1}^{d} \theta_k f_k, g_l} g_l
\eqqcolon
	\sum_{l=1}^{d} \hat{\theta_l} g_l
\end{equ}
Thus, we see that
\begin{equ}
\norm{\pi_{\theta} - \pi_{\theta'}}^2
=
\sum_{l=1}^{d} \pc{ \hat{\theta_l} - \hat{\theta'_l} }^2
\end{equ}
By Lemma \ref{lem:continuity of P} we see that the optimal controls $K_{\gamma}^*$ are continuous in $\gamma$. Thus we conclude that
\begin{equ}
\abs{\gamma-\gamma'} < \delta
\implies
\norm{\pi_{\theta_{\gamma}^*} - \pi_{\theta_{\gamma'}^*}} < \epsilon
\implies
\norm{\hat{\theta} - \hat{\theta'}} < \epsilon.
\end{equ}
Now, notice that we can also write the hat-coefficients as
\begin{equ}
\begin{bmatrix}
	\hat{\theta_1} \\
	\vdots \\
	\hat{\theta_d}
\end{bmatrix}
=
\begin{bmatrix}
	\dtp{f_1,g_1} & \cdots & \dtp{f_d,g_1} \\
	\vdots & \ddots & \vdots \\
	\dtp{f_1,g_d} & \cdots & \dtp{f_d,g_d}
\end{bmatrix}
\begin{bmatrix}
	\theta_1 \\
	\vdots \\
	\theta_d
\end{bmatrix},
\end{equ}
and this change-of-basis matrix is invertible by Assumption \ref{asm:linear indep}. Thus, denoting the above matrix by $A$, we conclude that
\begin{equ}
\norm{\hat{\theta} - \hat{\theta'}} < \epsilon
\implies
\norm{\theta-\theta'} < \norm{A^{-1}}\epsilon
\end{equ}
\end{proof}

%

\begin{restatable}{Rem}{RemarkOne}\label{rem:random features model}
2-layer ReLU neural networks (without bias and with a slightly modified architecture) trained in the random features regime satisfy the above assumptions. In this regime, we only train the weights of the output layer $\set{\theta_k} = \set{a_k} \cup \set{b_k}$. For the $i$'th component of the control, we can write
\begin{equ}
\pi_i(x)
=
\sum_{k=1}^{n} a_k \sigma(w_k\transpose x) + b_k \sigma(-w_k\transpose x),
\end{equ}
where $\sigma(x) = \max(0,x)$ is the rectified linear unit (ReLU), $x \in \Rr^n$ is the state, and $w_k \sim N(0, I_{n})$ are the first layer weights fixed at initialization. This type of modification with duplicated ReLU neurons has been considered before (see Appendix F of \cite{yehudai2020}).
\end{restatable}

\begin{proof}[Proof of Remark \ref{rem:random features model}]
Recall that $x \in \Rr^n$. For Assumption \ref{asm:linear indep}, we note that $\set{w_k}$ is a linearly independent set with probability 1. With this architecture, for every $x \in \Rr^n$ there will be exactly $n$ terms in the sum that are non-zero. Since $\set{w_k}$ forms a basis of $\Rr^n$, Assumption \ref{asm:linear functions in span} follows.
\end{proof}

\begin{Rem}\label{rem:homotopy assumption example}
The policy class $\set{f_k} = \set{1_{i,j}} \cup \set{F}$ for some globally Lipschitz function $F$ satisfies the above assumptions. Here $1_{i,j} \in \Rr^{m \times n}$ denotes the matrix with all 0's except for a $1$ in the $i,j$'th entry. This policy can be written as
\begin{equ}
u_t = Kx_t + \theta F(x_t)\,,
\end{equ}
where $K \in \Rr^{m \times n}$, $\theta \in \Rr$ are the parameters to be learned. This will be the policy class we consider in Section \ref{subsec:counterexample} to construct our counterexample.
\end{Rem}

\subsection{Counterexample with a Globally Lipschitz Policy}\label{subsec:counterexample}~\\

We now proceed to construct the aforementioned example where vanilla (fixed discount factor) policy gradients may result in convergence to a \emph{local} optimum. Importantly, the policy class of this example satisfies the assumptions needed for the convergence result provided in the subsequent section.  The lemmas and additional details for both subsections have been omitted for clarity, the full details can be found in Appendices \ref{apdx:counterex proofs} and \ref{apdx:homotopy proofs}.

We consider a simple 1-dimensional system ($X = U = \mathbb R, \gamma \in (0,1)$). Fix $\gamma$, we then choose $(A,B) = (0,1)$ and $(Q,R) = (1,R)$ as our system matrices (see \eref{eq:lqr dynamics}) for any $R \in (0,\gamma)$. Notice that this system is controllable and observable so there is an optimal control. For $a \in \mathbb R$ and $\delta>0$, we define the spike or tent function of width $2\delta$ centered at $a$:
\begin{equation*}
  \Lambda_{a,\delta}(x) \coloneqq
  \begin{cases}
  \frac{x-(a-\delta)}{\delta}\quad&\text{if } x\in \pq{a-\delta, a}\\
  \frac{(a+\delta)-x}{\delta}\quad&\text{if } x\in \pq{ a,a+\delta}\\
  0 &\text{else}
\end{cases}.
\end{equation*}
 Next, we define the control as
\begin{equ}\label{eq:cont example 1dcontrol}
u_{t}
=
\theta_0 x_t + \theta_1 F(x_t)
	\qquad \text{where }
F(x_t)
	\coloneqq
	\omega_0 |x_t| +
	\sum_{i=1}^3 \omega_i \Lambda_{a_i,\delta_i}(x_t),
\end{equ}
for a choice of parameters $\omega_0 \in (-1,0)$ and $\{a_i,\omega_i,\delta_i\}_{i=1}^3$ where $a_2,a_3 < \abs{a_1}$ that are \emph{fixed} during training of the model. We set
\begin{equ}\label{eq:cont example parameters}
  \omega_0 = -0.5 ,\,\,\,
  \delta_1 = 0.1 ,\,\,\,
  a_1 = -2 ,\,\,\,
  \omega_1 = 3 ,\,\,\,
  a_2 = 1.5 ,\,\,\,
  a_3 = 1.8 ,\,\,\,
  \delta_2 = \delta_3 = \delta ,\,\,\,
  \omega_2 = \omega_3 = 0.2,
\end{equ}
for a parameter $\delta$ to be chosen later.
The dynamics resulting from this choice of parameters are depicted in Figure \ref{fig:cont counterexample}. We note that the optimal policy for the LQR belongs to this policy class and corresponds to the parameters $(0, 0)$.

\begin{figure}
\centering
\includegraphics[width=0.75\textwidth]{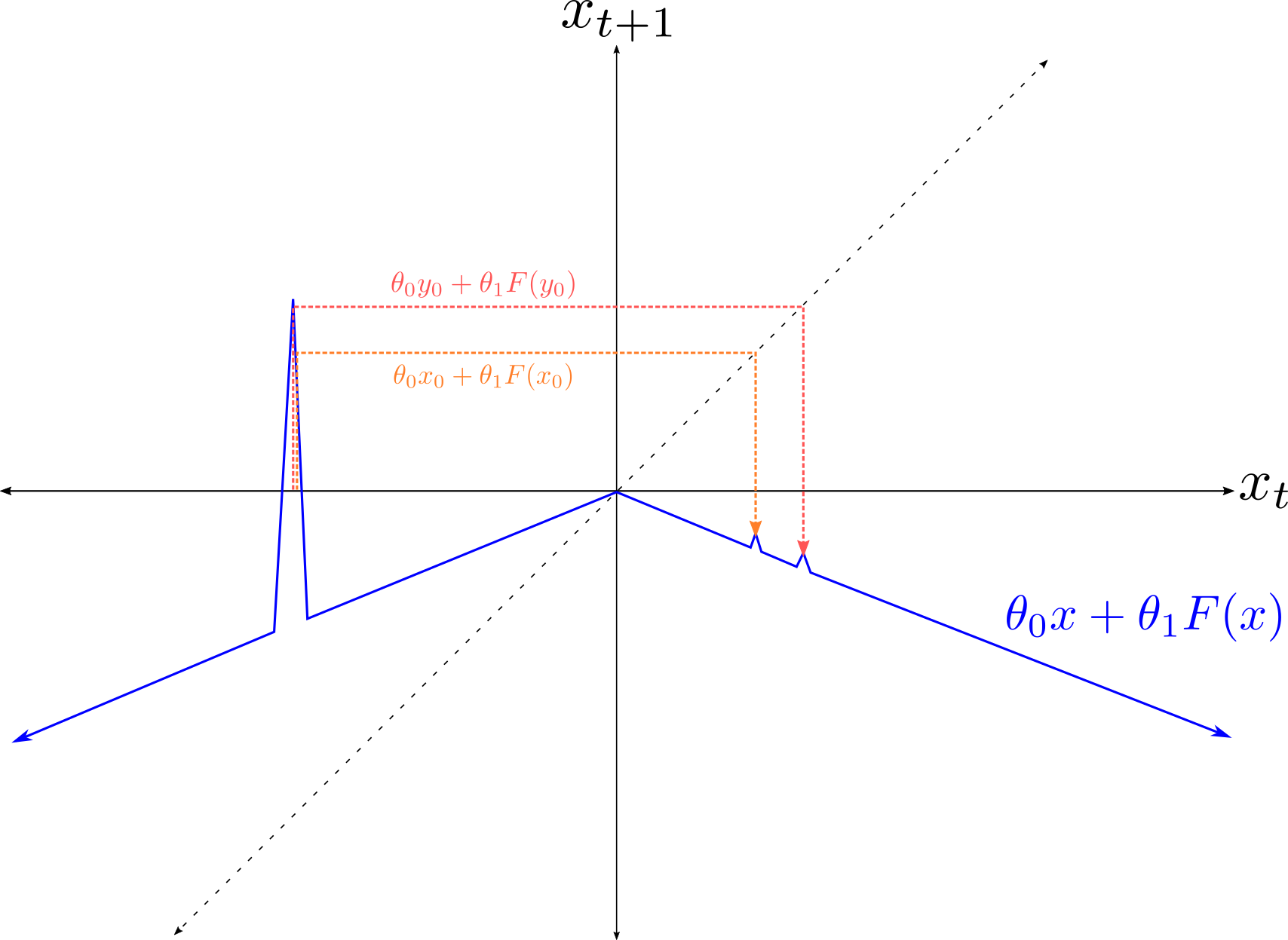}
\caption{A cobweb diagram for the first step of the LQR dynamics in the counterexample. The blue line represents the dynamics of the system at initialization $(\theta_0 = 0, \theta_1 = 1)$. Our initial distribution of states is heavily weighted around the two points ($x_0$ and $y_0$) that map to the two small tents on the right; this requires the points to lie somewhere in the domain of the spike on the left because all other states are mapped to negative numbers.}
\label{fig:cont counterexample}
\end{figure}

We choose the initial condition $(\theta_0, \theta_1) = (0,1)$ and define the distribution of the initial state as the following measure with support on $(-5,5)$
\begin{equ}\label{eq:cont example init state distrib}
\mu_0(x)
=
\frac{1-\epsilon}2 \delta_{x_0}+\frac{1-\epsilon}2 \delta_{y_0}+ \frac \epsilon {10},
\end{equ}
where $x_0 = \min F^{-1}(a_2)$ and $y_0 = \min F^{-1}(a_3)$,
\emph{i.e.}, we define $x_0$ and $y_0$ as the smaller of the two pre-images for $a_2$ and $a_3$, respectively.

\begin{restatable}{Prop}{CounterexampleProp}\label{prop:counterex local min}
Fix $\gamma \in (0,1)$ and $0 < R < \gamma$. For the LQR problem with $(A,B) = (0,1)$, $(Q,R) = (1,R)$, policy class \eref{eq:cont example 1dcontrol} with parameters chosen in \eref{eq:cont example parameters}, and initial state distribution $\mu_0$ from \eref{eq:cont example init state distrib}, there exists  $\delta > 0$ and $\epsilon > 0$ such that the point in parameter space $\theta = (0,1)$ is a local minimum of the LQR cost function.
\end{restatable}

To prove that the point $(0,1)$ is a local minimum of the cost function, and therefore a suboptimal fixed point of the policy gradient algorithm, we show that any infinitesimal change in $\theta_0$ and $\theta_1$ around $(\theta_0,\theta_1) = (0,1)$ will increase such cost function.

To understand why this is the case, consider first the trajectory of the points $x_0$ and $y_0$. It is easy to see that, by the design of the control, any small perturbation of the parameters that change the image of $x_0$ and $y_0$ will result in a state of larger absolute value, resulting in turn in a larger cost for those trajectories. On the other hand, it is also possible that a change in parameters actually lowers the cost for some trajectories starting in $x \in [-5,5] \setminus \set{x_0, y_0}$. Still, by our choice of $\epsilon$ it is possible to show that the increase in cost for $x_0, y_0$ outweighs the decrease. More specifically, the proof follows the following steps (the full details are presented in Appendix \ref{apdx:counterex proofs}):
\begin{enumerate}
\item For both $x_0, y_0$, we first prove that we can choose the parameter $\delta > 0$ (see \eref{eq:cont example parameters}) such that outside of a small cone in parameter space, we can bound the change in cost from below to show that it is always positive.

\item We then guarantee that we can pick the parameters of the model such that the cones described above corresponding to $x_0$ and $y_0$ only intersect at the origin.

\item Next, we show that the cost difference is still positive even if the infinitesimal change in $(\theta_0, \theta_1)$ belongs inside one cone and outside of the other.

\item Lastly, we bound the contribution of the uniform distribution term in $\mu_0$, showing that it is negligible with respect to any of the contributions to the cost function from the trajectories starting at $x_0$ or $y_0$.
\end{enumerate}

Qualitatively, we can plot the cost landscapes for our example around $\theta = (0,1)$ to see that this point indeed is a local minimum (see Figure \ref{fig:counterex cost landscape} in the appendix).
%
%
The example we have described is only one member of a family of possible counterexamples. The proofs we provide have used some of the specifics of our example; however, they can be very easily generalized to describe a set of parameters and constraints necessary for the point $\theta = (0,1)$ to correspond to a local minimum of the cost function. We only present this specific example as it is sufficient to demonstrate the failure of vanilla policy gradients in the LQR.

\subsection{Convergence Results}~\\

In this section we present our convergence result for policies of the form
\begin{equ}
\pi_{\theta}(x_t) = \sum_{k=1}^{d} \theta_k f_k(x_t)
\end{equ}
where each $f_k: \Rr^n \to \Rr^m$ is a non-linear, globally Lipschitz function and $\theta_k \in \Rr$ is a parameter to be learned. We prove our results under Assumptions \ref{asm:linear indep} and \ref{asm:linear functions in span}, restated here for convenience.

\AsmOne*

\AsmTwo*

We prove the convergence of policy gradient methods for the policy class of \eref{eq:policy class} by using the homotopy idea discussed in the introduction. Informally, our idea is that if one updates $\gamma$ by increasing it an infinitesimally small amount to $\gamma'$, one might expect that the optimal policies for the corresponding discount factors might not be too far from each other. Then if these policies, and importantly their parameterizations, are sufficiently close, policy gradient methods will yield convergence from the $\gamma$-optimal policy to the $\gamma'$-optimal policy. It's important to remember that one should not expect this behavior in general; we have the simplicity of the LQR to thank for the continuity of the optimal policy in the discount factor, but in more challenging settings such as Chess or Go, this may not be the case.\vspace{15pt}

\begin{algorithm}[H]
\SetAlgoLined
\KwResult{Optimal policy of the LQR problem}
$\hat \theta^* = \theta(0)$, $\{\gamma_n\}_{n=1}^N$\;
\For{$n \in \{1,\dots, N\}$}{
	$\gamma \leftarrow \gamma_n$ \;
	$\hat \theta^* \leftarrow \text{PG}(\hat \theta^*,\gamma)$ \;
}
\Return{$\hat \theta^*$}
\caption{Homotopy algorithm for LQR. Here $\text{PG}(\theta,\gamma)$ denotes the fixed point of the policy gradient algorithm for the given LQR problem with discount factor $\gamma$ and initial condition $\theta$.}
\label{alg:homotopy}
\end{algorithm}\vspace{15pt}

In other words, given some initial parameters $\theta(0)$ and an increasing sequence of discount factors $\gamma_n \in [0,1]$, we run the policy gradient algorithm until convergence to the optimal policy of the first discount factor. We then use this optimal policy as the initial policy for the system with the next discount factor, and let the policy gradient algorithm converge to the corresponding optimum. We continue iterating until we have converged to the optimal policy of the undiscounted LQR. To prove this algorithm converges, we need to prove that policy gradients yield convergence from the $\gamma$-optimal policy to the $\gamma'$-optimal policy if $\gamma$ and $\gamma'$ are sufficiently close, and that for any arbitrary initialization, the first iteration of the algorithm converges.

The proof is split into two main steps:
\begin{enumerate}
\item Prove convergence of arbitrary initializations to the $\gamma=0$ optimal policy.

\item Prove local convergence of policy gradient methods. In other words, we show that for discount factors $\gamma'$ and $\gamma$ sufficiently close, policy gradient methods yield convergence from the $\gamma$-optimal policy to the $\gamma'$-optimal policy.
\end{enumerate}
The first step brings policies into the ``homotopy'' regime, while the second step shows that we can iterate this process to go from the $\gamma=0$ optimal policy to the $\gamma=1$ (undiscounted) optimal policy. {We remark that it is also reasonable to ignore this first step and to initialize the policy to the zero function as is occasionally done in practice.}

In the statements of the theorems and their proofs, we recall that $\theta\in \mathbb R^d $ denotes the parameters of the model and $\theta_\gamma^*\in \mathbb R^d$ denotes the optimal parameters for the given value of $\gamma$, which exist and are unique by Assumptions \ref{asm:linear indep}, \ref{asm:linear functions in span}. 

\begin{restatable}[Convergence at $\gamma=0$]{Thm}{ThmConvOfFirstIter}
\label{thm:convergence of first iteration}
Let $\set{f_k}_{k=1}^{d}$ be a set of globally $\text{Lip}(f_k)$-Lipschitz continuous functions, and let $\rho_0$ (the distribution of the initial state) have full support on $\mathbb R^d$. Under Assumptions \ref{asm:linear indep}, \ref{asm:linear functions in span}, any random initialization of the parameters $\theta$ will converge to the optimal policy for discounted LQR with $\gamma = 0$.
\end{restatable}

We first note that the discount factor of this first iteration does not need to be $\gamma=0$; however, to make it easier to prove convergence of the first iteration, we can assume that $\gamma=0$. With such high myopia, the learning task becomes much easier, since the optimal policy for the LQR with $\gamma=0$ is simply the zero function.

The proof uses the fact that we can write a differential equation describing the evolution of the cost difference $C_\gamma(\theta(s)) - C_\gamma(\theta_\gamma^*)$ over time, where $C_\gamma(\theta)$ denotes the cost for the policy $\pi_\theta$ with discount factor $\gamma$. This results in an expression that is similar to gradient-domination bounds and from this the conclusion follows.

We now state the result about the local convergence of policy gradient methods in a neighborhood of the optimal policy for the chosen discount factor:

\begin{restatable}[Convergence Near $\gamma$-Optimality]{Thm}{ThmLocalConv} \label{thm:local nonlinear convergence}
Under the same conditions as Theorem \ref{thm:convergence of first iteration}, for any $\gamma \in (0,1)$ there exists $\delta>0$, $\lambda > 0$ such that for all initial conditions $\theta(0)$ with $\dtheta(0) \coloneqq \|\theta(0) - \theta_{\gamma}^* \| < \delta$,
\begin{equation*}
C_\gamma(\theta(s)) - C_\gamma(\theta_\gamma^*)
\leq
e^{-s\lambda}
\pc{ C_\gamma(\theta(0)) - C_\gamma(\theta_\gamma^*) },
\end{equation*}
and that $\lim_{s \to \infty} \theta(s) = \theta_\gamma^*$.
\end{restatable}

The main idea behind the proof of this result is to leverage a Taylor expansion of the cost function $C_{\gamma}(\,\cdot\,)$ around its minimum $\theta_{\gamma}^*$ for any fixed value of $\gamma \in (0,1)$. We show that by considering the second order Taylor expansion in a small neighborhood around optimality, the landscape of the policy gradient algorithm is locally convex. This allows to show that from any point within this neighborhood, policy gradient methods experience exponential convergence to the $\gamma$-optimal policy by applying Gr\"onwall's inequality.

Lastly, we want to show that by continuously updating the discount factor, we can ensure that we are always within a small neighborhood of the optimal policy for the \emph{next} discount factor, eventually reaching the undiscounted optimal.

\begin{Thm}[Convergence to Undiscounted Optimal]
\label{thm:homotopy for nonlinear case}
For any initial parameters $\theta_0$, Algorithm \ref{alg:homotopy} will converge to the undiscounted optimal policy $\theta^*$.
\end{Thm}

\begin{Cor}
\label{cor:homotopy for nonlinear case}
There exists an increasing sequence $\{\gamma_n\}_n$ of discount factors such that Algorithm \ref{alg:homotopy} converges to the global optimum of the undiscounted LQR.
\end{Cor}

\begin{proof}[Proof of \tref{thm:homotopy for nonlinear case}]
\tref{thm:convergence of first iteration} shows that any initialization of the policy will converge to the optimal policy of the $\gamma=0$ system. \tref{thm:local nonlinear convergence} {together with positive-definiteness of $H_\gamma$} then tells us that policy gradient methods experience local exponential convergence near optimality. Corollary \ref{cor:continuity of K} and Lemma \ref{lem:homotopy} tell us that for $\abs{\gamma'-\gamma}$ sufficiently small, $\theta_{\gamma}^*$ will be in close enough to $\theta_{\gamma'}^*$ for \tref{thm:local nonlinear convergence} to apply. One can iterate this process until reaching the undiscounted optimal policy.
\end{proof}

\section{Discussion and Conclusion} \label{sec:Conclusion}

This paper provides a convergence guarantee for model-based policy gradient methods in the setting of the LQR with the policy class $\set{\sum_{k=1}^{d} \theta_k f_k(x) \given \theta_k \in \Rr}$ contingent on a few assumptions on the expressiveness and parametrization of the policy class. Unlike previous works, we adopt a new approach for proving convergence, namely, the homotopy-based approach which, to the best of our knowledge, has not been used before in theoretical reinforcement learning. Using a discount factor to guarantee finite costs and well-defined updates can also be applied to solving issues involving a chaotic policy, such as the one arising naturally from piece-wise linear policies such as the tent map.

Furthermore, we provide an example illustrating a situation where expanding the policy class in the LQR case leads to local minima of vanilla policy gradient algorithms. This issue is resolved by applying the homotopic variant of policy gradients developed in this paper. 

\subsubsection*{Limitations.} It is important to note that our work considers an idealized, model-based and infinitesimal-stepsize approach. To generalize our results to the model-free perspective, one could proceed in a similar fashion to \cite{fazel2018} by showing that when the roll-out is sufficiently long, the cost function and covariance of the state trajectory can be accurately approximated  and that with enough samples, the true gradient can be approximated within a desired accuracy. Some other interesting avenues include relaxing the assumption of continuous updating, \emph{i.e.}, to consider the stochastic approximation problem resulting from the real-life setting of finite samples and finite step-size and establishing quantitative bounds for the homotopy method.

Furthermore, the approach proposed in this work results in a significant increase in the computational cost of training when compared to the vanilla policy gradient algorithm. This results in a trade-off between computational cost and convergence guarantees for solving the control problem at hand. While this trade-off may not be of immediate interest in practice, what we propose is a new paradigm for studying the convergence of reinforcement learning algorithms.

\subsubsection*{Future Work.} The problem of extending our proofs to establish convergence for the policy class of neural networks remains open. We expect that it should be possible to extend our result to linearized neural networks (such as the Neural Tangent Kernel regime) but we leave this result for future research.

\subsubsection*{Societal Impact.}

In the much bigger picture, our work is related to the issue of bias in machine learning. The common practice of seeding reinforcement learning agents with human training data can be problematic when the human-generated data contains implicit biases that we do not want in the agent --- almost all data associated with humans will contain such biases. Our work specifically shows that we can do away with some assumptions on the initial knowledge of reinforcement learning agents (\eg assuming the initial policy of the LQR is stabilizing), and thus, in applicable settings, do away with any dependence on human-generated data. We choose to include this reminder because it is imperative to keep this question in mind when designing machine learning models with the goal of bettering human lives. A simple example is the autonomous vehicle trolley problem. In some cultures, it may be preferable to crash into one individual over the other, and thus a model trained in one country may have disastrous consequences when deployed in another.

\subsubsection*{Acknowledgements} The authors thank Jianfeng Lu for many insightful discussions. CXC acknowledges the PRUV program at Duke University for partial support. AA acknowledges the support of the Swiss National Science Foundation through the grant P2GEP2-17501 and
the NSF grant DMS-1613337.

\bibliographystyle{plain}
\bibliography{references}
\newpage

\appendix

\section{Numerical experiments}\label{apdx:figures}

\subsection{Running Algorithm~\ref{alg:homotopy}}
\begin{center}
\includegraphics[width=\textwidth]{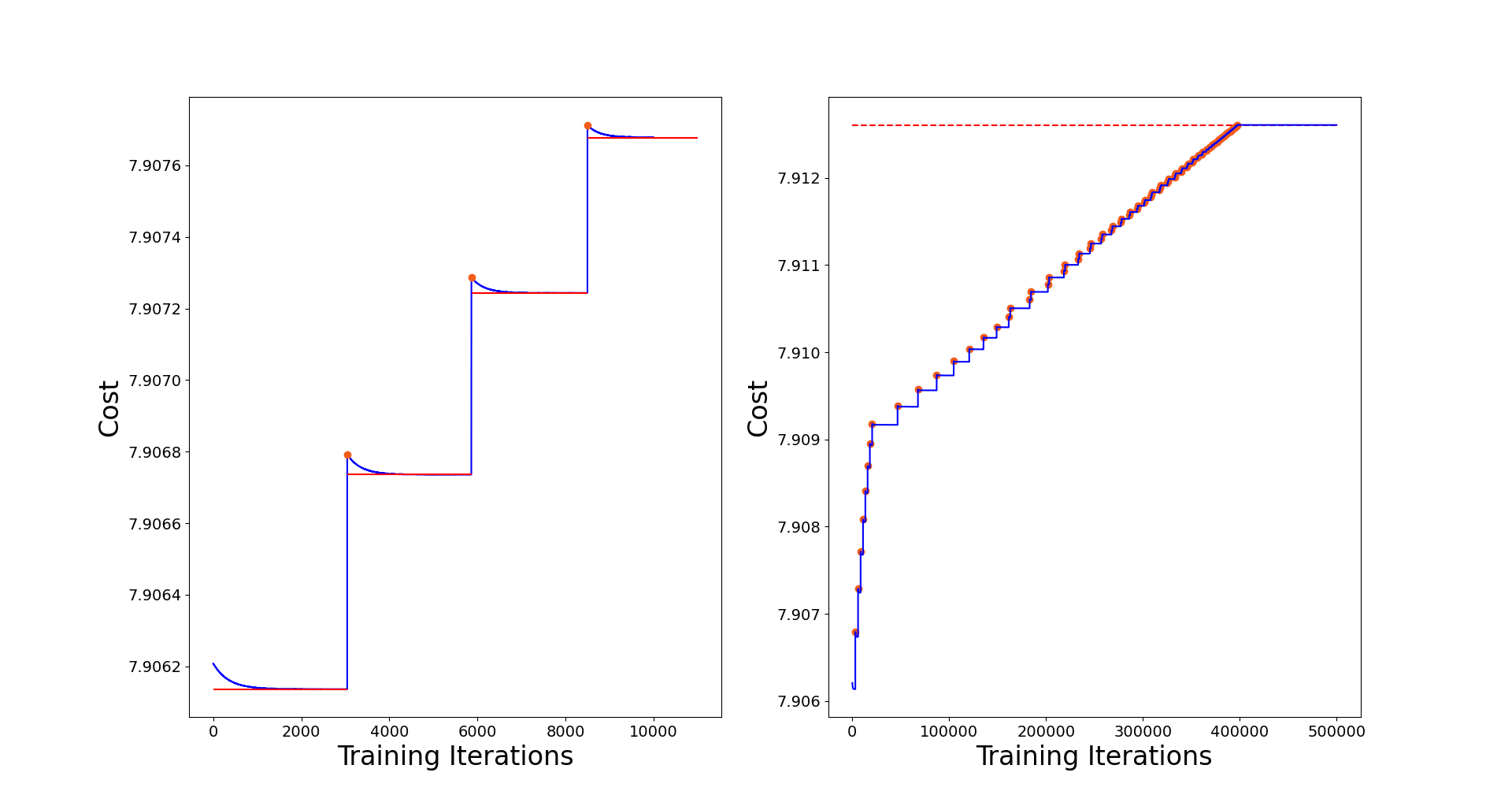}
\captionof{figure}{Evolution of the cost during Algorithm~\ref{alg:homotopy} as a function of training iterations. The blue line is the cost during training of the homotopy algorithm in a similar setting to the one described in Section \ref{subsec:counterexample}. The orange dots represent a point at which $\gamma$ was updated. The red lines are the optimal costs for the various values of $\gamma$. The graph on the left is a close-up of the cost between iterations $0$ and $10^5$, the one on the right represents the full training process. The dashed red line on the plot on the right is the optimal cost for $\gamma = 1$.
}
\label{fig:homotopy training cost}
\end{center}

We run Algorithm \ref{alg:homotopy} to solve the LQR problem \eqref{eq:cost}, \eqref{eq:lqr dynamics} in  dimension $n=m=1$ with $(A,B) = (0.1,1)$ and $(Q,R) = (1,0.1)$ with the parametric policy class described in Section~\ref{subsec:counterexample} and setting  $\delta = 5\text{e-}4$. We do so by writing a policy using PyTorch and updating the parameters via gradient descent with a stepsize of $1\text{e-}3$.
We increased $\gamma$ by $0.02$ whenever the absolute values of both partial derivatives with respect to the parameters $\theta_0$, $\theta_1$ were less than a tolerance hyperparameter, in our case set to $1\text{e-}4$. Note that the initial states $x_0$ and $y_0$ described in Section \ref{subsec:counterexample} can be explicitly calculated to be $-117/59$ and $-588/295$, respectively.

We trained the model for a total of 500000 episodes. During each episode, we interacted with the LQR system for a fixed number of steps, in our case, 5. Since the value for $\epsilon$ determined by our model is negligibly small (\ie $\epsilon < 1\text{e-}12$), we simultaneously simulate the roll-outs from $x_0$ and $y_0$ during training and update using the total sum of discounted costs experienced in both roll-outs. We update the parameters after every episode. To highlight the homotopy algorithm, we initialize our policy with the parameters $\theta = (0,0)$ which correspond to the $\gamma=0$ optimal policy.
\vspace*{\fill}

\subsection{Counterexample}

\captionsetup[subfigure]{labelformat=parens,labelsep=space}
\begin{center}
	\captionsetup{type=figure}\addtocounter{figure}{-1}
	\begin{subfigure}{\textwidth}
		\includegraphics[width=\linewidth]{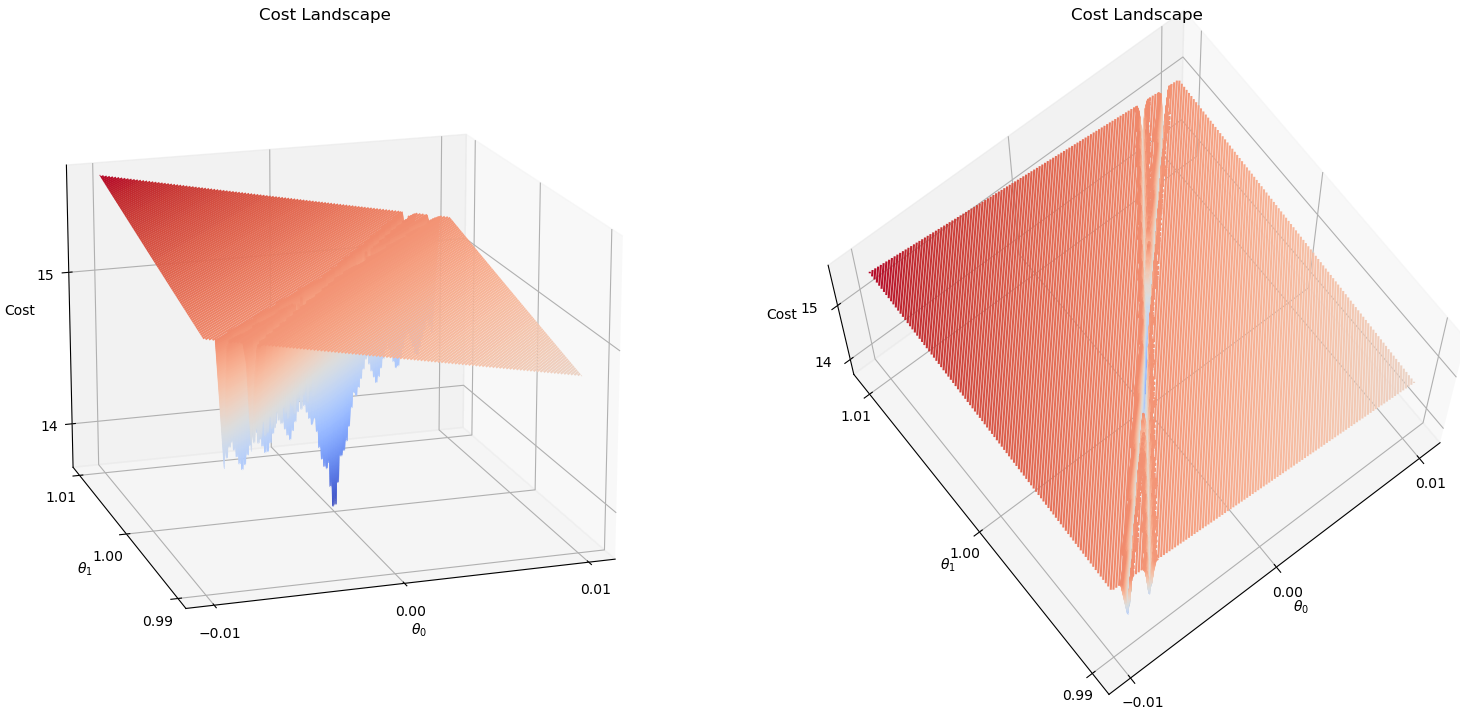}
		\caption{Notice, in the top-right figure, the two ``good'' directions in parameter space.}
	\end{subfigure}
\vskip\baselineskip
	\begin{subfigure}{\textwidth}
		\includegraphics[width=\linewidth]{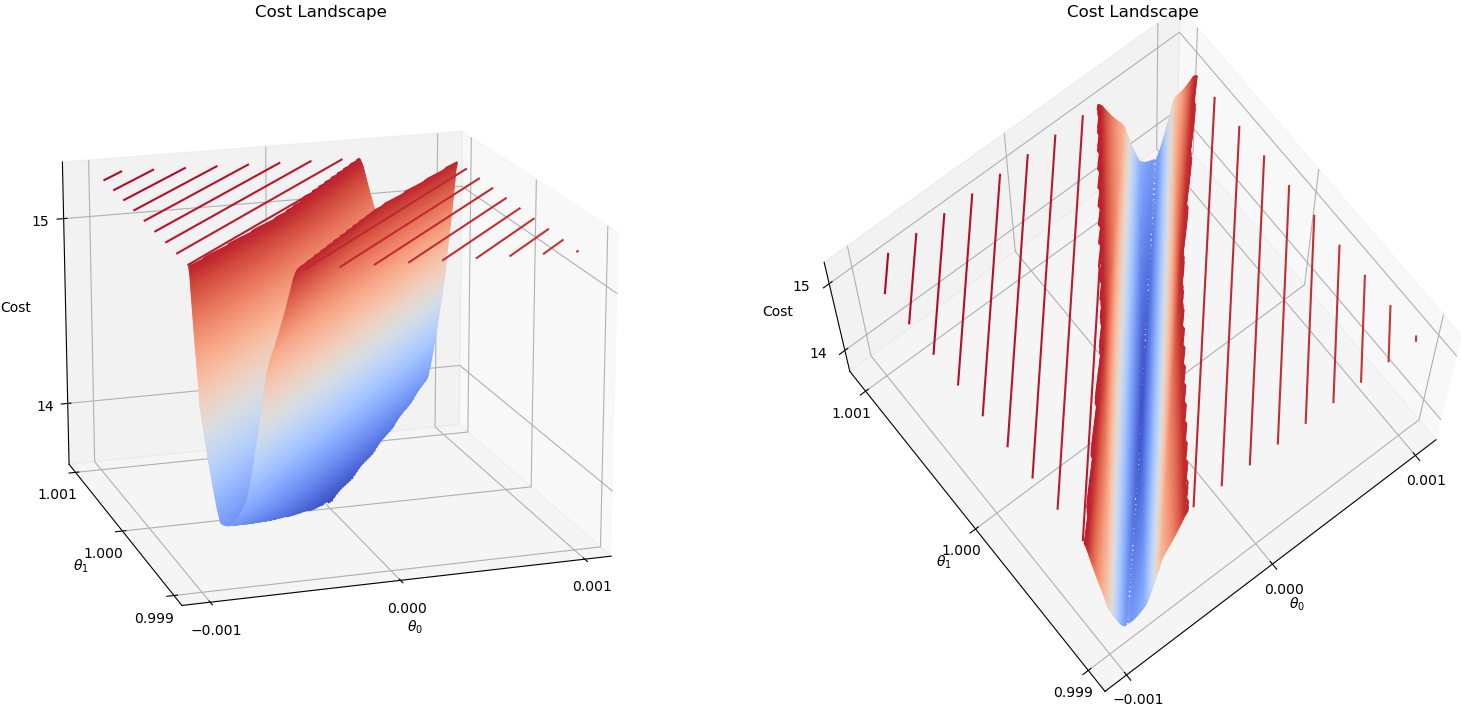}
		\caption{Close-up of the above plots.}
	\end{subfigure}
\captionof{figure}{Landscape of the cost function from the example of Section \ref{subsec:counterexample} in a neighborhood around $\theta = (0,1)$. The local minimum at $(0,1)$ can be clearly seen in these plots.
}
\label{fig:counterex cost landscape}
\end{center}

We compute the landscape of the cost function to verify numerically the result of Proposition~\ref{prop:counterex local min} and represent the lanscape in Figure~\ref{fig:counterex cost landscape}. These results were generated by evaluating the cost experienced by the policy with parameters $(\theta_0, \theta_1)$ in a neighborhood of $(0,1)$ for the choice of parameters in Section~\ref{subsec:counterexample} and setting  $\epsilon = 0$ and $\delta = 5\text{e-}4$.
The cost of a policy was evaluated by rolling out the policy from both $x_0$ and $y_0$ for 5 steps and computing the discounted sum of the instantaneous costs. Plots were made using the contour plot method from matplotlib.pyplot.

\vspace*{\fill}

\pagebreak

\section{Proofs for the Counterexample}\label{apdx:counterex proofs}

Before stating and proving our results, we give a roadmap for the structure of the proof. In Lemma~\ref{lem:bound cost diff outside cone} we show that the cost associated to trajectories starting at $x_0$ or $y_0$ is lower-bounded locally by the $\ell_1$ norm of the perturbation of the parameters $\|d \theta\|_1$ outside of 2 cone in parameter space (one for each initial condition). Then, in Lemma~\ref{lem:bad change is stronger than good change} we show that a similar lower bound -- with a reduced constant -- holds when the perturbation $d\theta$ lies in one such cone (\emph{e.g.} the one associated to $x_0$) but outside of the other (\emph{e.g.} the one associated to $y_0$), \ie in such regions the positive contribution to the cost of one of the trajectories outweights the one of the other. Combining this with Lemma~\ref{lem:cones dont overlap} showing that for a choice of parameters the cones do not overlap, we obtain in Corollary~\ref{cor:cones} that the \emph{joint} contribution to the cost function associated to trajectories starting at $x_0, y_0$ is lower-bounded locally by (a constant times) the $\ell_1$ norm of the perturbation of the parameters $\|d \theta\|_1$. We conclude by further combining this with Lemma~\ref{lem:bound contribution of uniform measure}, showing that the contribution to the calculation of the uniform measure is negligible, so that the lower bound stated above holds for $\mu_0$.

Throughout, we define
\begin{equ}
  C[x_0](\theta) := \sum_{t=0}^{\infty} \gamma^t
	\pc {x_t(\theta)^\intercal Qx_t(\theta) + \pi_\theta(x_t(\theta))^\intercal R \pi_\theta(x_t(\theta))}
\end{equ}
where $x_t(d \theta) \in \mathbb R^n$ denotes the trajectory of the LQR with control $\pi_{\theta + d\theta}$ for $\theta=(0,1)$ and initial condition $x_0$.

\subsubsection*{Notation} We will use $z_t$ to denote the $t$'th state of the unperturbed trajectory and we will use $z_t(d\theta)$ for the corresponding state of the perturbed trajectory (\ie generated by policy $\pi_{\theta + d\theta}$).

\begin{Lemma}\label{lem:bound cost diff outside cone}
Let $\theta = (0,1)$ and let $z_0 \in (x_0,y_0)$ defined in Section~\ref{subsec:counterexample}. For any $\alpha > 0$, there exists $\delta > 0$, $\KK >0$ such that
for any $i \in \{2,3\}$ and $d\theta = (d\theta_0, d\theta_1)$  satisfying
\begin{equ}
\abs{z_0 d\theta_0 + F(z_0) d\theta_1} \geq \alpha \norm{d\theta}_1,
\end{equ}
the cost difference between the contribution to the cost function using the perturbed and unperturbed policies for trajectories starting from $z_0$ is strictly positive and bounded from below:
\begin{equ}
C[z_0](\theta + d\theta) - C[z_0](\theta)
\geq
\frac{3\KK}{4\delta} \norm{d\theta}_1\,.
\end{equ}
\end{Lemma}

\begin{proof}[Proof of Lemma \ref{lem:bound cost diff outside cone}]
Throughout, let $z,a,\omega$ be generic variables that can be replaced by either $(x,a_2,\omega_2)$ or $(y,a_3,\omega_3)$. We first consider how the cost of a trajectory changes for any infinitesimal variation of the parameters $d\theta = (d\theta_0, d\theta_1)$. To do so we calculate the first order corrections (in $d \theta$) to the first two timesteps of the trajectory $z_t(d\theta)$ starting at $z_0$. For the first timestep, since along the unperturbed trajectory we have $z_1 = F(z_0) = a$, we have
\begin{equs}
z_1(d\theta)
&=
	d\theta_0 z_0 + (1+ d\theta_1)F(z_0) \\
	&=
	z_1 + \pc{a d\theta_1 + z_0 d\theta_0},\label{e:firstord}
\end{equs}
where here and throughout this section $z_t = z_t(0)$ denotes the unperturbed trajectory. For the second timestep,
\begin{equs}
z_2(d \theta)
&=
	d\theta_0 z_1(d\theta) + (1+d\theta_1)\pc{
		\omega_0 \abs{z_1(d\theta_1)} +
		\omega \Lambda_{a, \delta}(z_1(d\theta_1))
	} \\
	&=
	d\theta_0 z_1 +
	(1+d\theta_1) \bigg(
		\omega_0 \abs{z_1} + \omega_0 (a d \theta_1 + z_0 d \theta_0)
		\\ &\qquad +
		\omega \pc{
			\Lambda_{a,\delta}(z_1) +
			\Lambda_{a,\delta}'(z_1)( a d \theta_1 + z_0 d \theta_0 )
		}
	\bigg) \\
	&=
	a d\theta_0 +
	(1+d\theta_1)\pc{
		z_2 + \omega_0 (a d \theta_1 + z_0 d \theta_0) -
		\frac{\omega}{\delta} \abs{ a d \theta_1 + z_0 d \theta_0 }
	} \\
	&=
	z_2 + \pc{
		z_2 d\theta_1 + a d\theta_0 +
		\omega_0 ( a d \theta_1 + z_0 d \theta_0 ) -
		\frac{\omega}{\delta} \abs{ a d \theta_1 + z_0 d \theta_0 }
	} \\
	&\eqqcolon
	z_2 + dz_2
\end{equs}
where $dz_2 := z_2 d\theta_1 + a d\theta_0 +
\omega_0 ( a d \theta_1 + z_0 d \theta_0 ) -
\frac{\omega}{\delta} \abs{ a d \theta_1 + z_0 d \theta_0 }$ is the first order correction to the second timestep of the trajectory.
We then note that since $z_2\in (-1,0)$ due to our choice of parameters, the future cost from $z_2$ is given by
\begin{equs}
\gamma^2 C[z_2](\theta+d\theta)
&=
	\gamma^2 \sum_{t=0}^{\infty} \gamma^t \pc{ Q(z_{t+2})^2 + R(u_{t+2})^2 }
	=
	\gamma^2 \sum_{t=0}^{\infty} \gamma^t \pc{ Q(z_{t+2})^2 + R(z_{t+3})^2 } \\
&=
	\gamma^2 \pc{
		Q(z_2)^2 + \sum_{t=0}^{\infty}
		\gamma^t \pc{ \gamma Q (z_{t+3})^2 + R(z_{t+3})^2 } } \\
&=
	\gamma^2 \pc{
		Q(z_2)^2 +
		(\gamma Q+R)\sum_{t=0}^{\infty} \gamma^t
		(\pc{d\theta_0 + \omega_0(1 + d\theta_1)}^{t+1} z_2)^2
	} \\
&=
	\gamma^2 (z_2)^2 \pc{
		Q +
		\frac{(\gamma Q+R) (\omega_0 + (d\theta_0 + \omega_0 d\theta_1))^2}
		{1- \gamma \pc{\omega_0 + (d\theta_0 + \omega_0 d\theta_1)}^2 }
	}
\end{equs}
Since $z_2(d\theta) \in (-1,0)$ by construction as well, the above expression also holds for $z_2(d\theta)$. Therefore, since $Q=1$ and $1 + R \omega_0^2 > 0$, the cost of trajectory from $z_2$ onward will vary (to first order in $d \theta$) by
\begin{equs}
\gamma^2 z_2(d\theta)^2
\frac{
	1 + R(\omega_0 + (d\theta_0 + \omega_0 d\theta_1))^2
}{
	1- \gamma \pc{\omega_0 + (d\theta_0 + \omega_0 d\theta_1)}^2
}
&=
	\gamma^2 \frac{
		1 + R(\omega_0 + (d\theta_0 + \omega_0 d\theta_1))^2
	}{
		1- \gamma \pc{\omega_0 + (d\theta_0 + \omega_0 d\theta_1)}^2
	} \pc{
		\pc{z_2}^2 + 2z_2 d z_2
	} \\
&=
	\frac{
		\gamma^2 \pc{ 1 + R \omega_0^2 }
	}{
		1-\gamma \omega_0^2
	}
	\pc{
		\pc{z_2}^2 + 2z_2 d z_2
	}
	\\ &\qquad +
	\frac{
		2 \gamma^2 \omega_0 (\gamma + R )
	}{
		\pc{1 - \gamma \omega_0^2}^2
	} (d\theta_0 + \omega_0 d\theta_1) \pc{z_2}^2 \\
&=
	\frac{\gamma^2}{1- \gamma \omega_0^2 } \Bigg(
		\pc{ 1 + R \omega_0^2 } \pc{z_2}^2
		+ \pc{ 1 + R \omega_0^2 } 2z_2 d z_2
		\\ &\qquad +
		\frac{
			2 \omega_0 (\gamma + R )
		}{
			1-\gamma \omega_0^2
		} (d\theta_0 + \omega_0 d\theta_1) \pc{z_2}^2
	\Bigg)
\end{equs}
Inserting the expression of $d z_2$,
\begin{equs}
\text{RHS}
&=
	\frac{\gamma^2 \pc{ 1 + R \omega_0^2 }}{1- \gamma \omega_0^2 } \bigg(
		\pc{z_2}^2 + 2z_2 \bigg(
			- \frac{\omega}{\delta} \abs{ a d \theta_1 + z_0 d \theta_0 }
			+ z_2 d\theta_1 + a d\theta_0
			\\ &\qquad +
			\omega_0 ( a d \theta_1 + z_0 d \theta_0 ) +
			\frac{
				\omega_0 (\gamma + R )
			}{
				\pc{ 1 + R \omega_0^2 }(1-\gamma\omega_0^2)
			} z_2(d\theta_0 +\omega_0 d\theta_1)
		\bigg)
	\bigg)\label{e:RHS} \\
&\geq
	\frac{\gamma^2\pc{ 1 + R \omega_0^2 }}{1- \gamma \omega_0^2 } \pc{
		\pc{z_2}^2 - 2z_2
		\frac{3\omega \alpha}{4 \delta} \norm{ d \theta }_1
	} \\
&=
	\frac{\gamma^2\pc{ 1 + R \omega_0^2 }}{1- \gamma \omega_0^2 } \pc{
		\pc{z_2}^2 + 2\abs{z_2}
		\frac{3\omega \alpha}{4 \delta} \norm{ d \theta }_1
	},
\end{equs}
since $z_2$ is negative. The inequality comes from the fact that whenever $\abs{z_0 d\theta_0 + a d\theta_1} \geq \alpha \|d \theta\|_{1}$  we can choose $\delta>0$ sufficiently small such that
\begin{equs}
\abs*{\frac{\omega \abs{a d\theta_1 + z_0 d\theta_0}}{4\delta}}
&\geq
\abs*{\frac{\omega \alpha \norm{d\theta}_1}{4\delta}} \\
&\geq
4\max \Bigg(\Bigg|
	z_2 d\theta_1 + a d\theta_0\Bigg|,\Bigg|
	\omega_0 ( a d \theta_1 + z_0 d \theta_0 )\Bigg|\,,\label{e:delta}
	\\ &\qquad\qquad\qquad
	\Bigg|\frac{
		\omega_0 (\gamma + R )
	}{
		\pc{ 1 + R \omega_0^2 }(1-\gamma\omega_0^2)
	} z_2 (d\theta_0 + \omega_0 d\theta_1)
\Bigg|,\\
&\qquad\qquad\qquad
\Bigg|\frac{
   (\gamma + R )a |a d\theta_1+z_0 d \theta_0|
}{
z_2 \gamma^2 \pc{ 1 + R \omega_0^2 }/(1-\gamma\omega_0^2)
}
\Bigg|\Bigg)
\end{equs}
Bounding the right-hand-side by triangle inequality, the above is satisfied when
\begin{equ}
\delta
\leq
\frac{
	\abs{\omega \alpha}
}{
	4\max\pg{M_0,M_1}
},
\end{equ}
where
\begin{equs}
M_0
&=
4 \max\pg{
	\abs{z_2},\,
	\abs{\omega_0 a},\,
	\abs*{\frac{
			\abs{z_2} \omega_0^2 (\gamma + R )
		}{
			\pc{ 1 + R \omega_0^2 }(1-\gamma\omega_0^2)
		}
	},\,
	\abs*{\frac{
			a^2 (\gamma + R)(1-\gamma\omega_0^2)
		}{
			\abs{z_2} \gamma^2 (1+R\omega_0^2)
		}
	}
} \\
M_1
&=
4 \max\pg{
	\abs{a},\,
	\abs{\omega_0 z_0},\,
	\abs*{\frac{
			z_2 \omega_0 (\gamma + R )
		}{
			\pc{ 1 + R \omega_0^2 }(1-\gamma\omega_0^2)
		}
	},\,
	\abs*{\frac{
			az_0 (\gamma+R)(1-\gamma\omega_0^2)
		}{
			z_2\gamma^2 (1+R\omega_0^2)
		}
	}
}.
\end{equs}
In other words, we have shown that for all $d\theta \in \set*{(d\theta_0,d\theta_1) \in \Rr^2 \given \abs{z_0 d\theta_0 + F(z_0) d\theta_1} \geq \alpha \norm{d\theta}_1}$,
\begin{equs}
C[z_0](\theta + d\theta) - C[z_0](\theta)
&=
	\pc{z_0^2 + (\gamma + R ) z_1(d\theta)^2 + \gamma R z_2(d\theta)^2 + \gamma^2 C[z_2(d\theta)](\theta + d\theta)}
	\\ &\qquad -
	\pc{z_0^2 + (\gamma + R ) z_1^2 + \gamma R z_2^2 + \gamma^2 C[z_2](\theta)} \\
&\geq
	\gamma^2 \pc{C[z_2(d\theta)](\theta + d\theta) - C[z_2](\theta)} - 2(\gamma + R )\abs{z_1} \pd{z_1(d\theta) - z_1} \\
  &=\gamma^2 \pc{C[z_2(d\theta)](\theta + d\theta) - C[z_2](\theta)} - 2(\gamma + R )a \pc{a d\theta_1 + z_0 d\theta_0}\\
&\geq
	3\frac{
		2\abs{z_2} \gamma^2\pc{ 1 + R \omega_0^2 }\omega \alpha
	}{
		\pc{1 - \gamma \omega_0^2 } 4\delta
	}
	\norm{ d \theta }_1  \label{eq:cont example lower bound cost diff} \\
&\eqqcolon
	\frac{3\KK}{4\delta} \norm{ d\theta }_1.
\end{equs}
for  \begin{equ}\label{e:k}
\KK
:=
\frac{
	2\abs{z_2} \gamma^2\pc{ 1 + R \omega_0^2 }\omega \alpha
}{
	1 - \gamma \omega_0^2
}
\end{equ}
where the first inequality comes from our assumption that $d\theta$ lies outside of the cone in parameter space ($\abs{z_0 d\theta_0 + F(z_0) d\theta_1} \geq \alpha \norm{d\theta}_1$) and the last one results from our choice of $\delta$ small enough from \eqref{e:delta}.
\end{proof}

Based on the above, the only variation of the parameters that does not increase the cost of the policy is when $d\theta$ is inside the cone
\begin{equ}
\abs{z_0 d\theta_0 + F(z_0) d\theta_1} \leq \alpha \norm{d\theta}_1,
\end{equ}
for $z_0 \in \set{x_0,y_0}$.
For the statement of the next lemma, for a choice of $z_0 \in \set{x_0,y_0}$ we define $\bar z_0 \in \set{x_0,y_0}\setminus z_0$ as the \emph{other} initial condition and by $\bar z_t(d\theta)$ the corresponding perturbed trajectory.

\begin{Lemma}\label{lem:bad change is stronger than good change}
Suppose $d\theta \in \set{\abs{z_0 d\theta_0 + F(z_0) d\theta_1} \leq \alpha \norm{d\theta}_1} \cap \set{\abs{\bar z_0 d\theta_0 + F(\bar z_0) d\theta_1} \geq \alpha \norm{d\theta}_1}$. For the same $\delta, \KK>0$ as determined by Lemma \ref{lem:bound cost diff outside cone}, the cost difference is strictly positive and can be bounded by
\begin{equ}
\int_{\Rr} (C[x](\theta + d\theta) - C[x](\theta))  \mu_{0,\tiny\text{pp}}(\d x)
\geq
\frac{1-\epsilon}{2} \frac{\KK}{12\delta} \norm{d\theta}_1\,.
\end{equ}
 Here $\mu_{0,\tiny\text{pp}} = \frac{1-\epsilon}{2}(\delta_{x_0} +\delta_{y_0})$ denotes the pure-point, or discrete, part of $\mu_0$.
\end{Lemma}

\begin{proof}[Proof of Lemma \ref{lem:bad change is stronger than good change}]
{The proof proceeds similarly to the proof of Lemma \ref{lem:bound cost diff outside cone}. We start by bounding from below the variation of the cost due to the trajectory $z_t$ starting at $z_0$ with $\abs{z_0 d\theta_0 + F(z_0) d\theta_1} \leq \alpha \norm{d\theta}_1$. Recalling the expression for the first-order variation of the cost \eref{e:RHS} from the proof of Lemma \ref{lem:bound cost diff outside cone} and that $z_2<0$ by our choice of parameters, we bound the negative part of the contribution to the cost function for $t\geq2$ as:}
\begin{equs}
(\gamma^2 C[z_2(d\theta)](\theta + d\theta) &- \gamma^2 C[z_2](\theta))_-\\&
=
	\frac{
		\gamma^2\pc{ 1 + R \omega_0^2}
	}{
		1- \gamma \omega_0^2
	} 2z_2 \bigg(
		z_2 d\theta_1 + a d\theta_0 +
		\omega_0 ( a d \theta_1 + z_0 d \theta_0 )
		\\ &\qquad -
		\frac{\omega}{\delta} \abs{ a d \theta_1 + z_0 d \theta_0 }
		+
		\frac{
				\omega_0 (\gamma + R )
		}{
			\pc{ 1 + R \omega_0^2 }(1-\gamma\omega_0^2)
		} z_2(d\theta_0 +\omega_0 d\theta_1)
	\bigg)_+ \\
&=
	\frac{
		\gamma^2\pc{ 1 + R \omega_0^2}
	}{
		1- \gamma \omega_0^2
	} \frac{\omega_2}{\delta} 2\abs{z_2}
	\abs{ a d \theta_1 + z_0 d \theta_0 }
	\\ &\qquad +
	\frac{
		\gamma^2\pc{ 1 + R \omega_0^2}
	}{
		1- \gamma \omega_0^2
	} 2z_2 \bigg(
		z_2 d\theta_1
			+
			a d\theta_0 +	\omega_0 ( a d \theta_1 + z_0 d \theta_0 )
		\\ &\qquad +
		\frac{
				\omega_0 (\gamma + R )
		}{
			\pc{ 1 + R \omega_0^2 }(1-\gamma\omega_0^2)
		} z_2(d\theta_0 +\omega_0 d\theta_1)
	\bigg)_+ \\
&\geq
	\frac{
		\gamma^2\pc{ 1 + R \omega_0^2}
	}{
		1- \gamma \omega_0^2
	} 2z_2 \bigg(
		z_2 d\theta_1
			+
			a d\theta_0 +	\omega_0 ( a d \theta_1 + z_0 d \theta_0 )
		\\ &\qquad +
		\frac{
			\omega_0 (\gamma + R )
		}{
			\pc{ 1 + R \omega_0^2 }(1-\gamma\omega_0^2)
		} z_2(d\theta_0 +\omega_0 d\theta_1)
	\bigg)_+
\end{equs}
where $(b)_- := \min(0,b), (b)_+:=\max(0,b)$ denote the positive and negative part of $b\in \mathbb R$ respectively.
{We can assume that the quantity on the LHS above is negative; if it were positive, then the proof is done since we are aiming to lower bound the total difference, so we could ignore this term if it were positive.} Since we are using the same $\delta$ as chosen in Lemma \ref{lem:bound cost diff outside cone}, recalling the definition of $\KK>0$ and that $z_2 < 0$ we obtain
\begin{equ}
\gamma^2 \pc{C[z_2(d\theta)](\theta + d\theta)| - C[z_2](\theta)}_-
\geq
	-\frac{
		2\gamma^2\abs{z_2}\pc{1+ R \omega_0^2 }
	}{
		1- \gamma \omega_0^2
	} \abs*{\frac{\omega_2 \alpha \norm{d\theta_1}}{4\delta}}
=
	-\frac{\KK}{4\delta} \norm{d\theta}_1\,.
\end{equ}
Then, recalling that $z_1 = F(z_0) = a$ and $0 < R < \gamma$, the total cost difference from the initial point $z_0$ (to first order) can then be bounded as follows for $\|d \theta \|$ sufficiently small:
\begin{equs}
C[z_0](\theta + d\theta) - C[z_0](\theta)&\geq (C[z_0](\theta + d\theta) - C[z_0](\theta))_-\\
&=
	\big(\pc{
		z_0^2 + (\gamma + R ) z_1(d\theta)^2 + \gamma R z_2(d\theta)^2
		+ \gamma^2 C[z_2(d\theta)](\theta + d\theta)
	} \\ &\qquad \qquad-	\pc{
		z_0^2 + (\gamma + R ) z_1^2 + \gamma R z_2^2
		+ \gamma^2 C[z_2](\theta)
	}\big)_- \\
&\geq
	-(\gamma + R ) \pd{z_1(d\theta)^2 - z_1^2}
	+\gamma R \pc{z_2(d\theta)^2 - z_2^2}_-
	\\ &\qquad \qquad+
	\gamma^2 \pc{C[z_2(d\theta)](\theta + d\theta) - C[z_2](\theta)}_- \\
&\geq
	-2(\gamma + R )\abs{z_1} \pd{z_1(d\theta) - z_1} \\&\qquad\qquad +
	2 \gamma^2 \pc{C[z_2(d\theta)](\theta + d\theta) - C[z_2](\theta)}_- \\
&\geq
	-2 (\gamma + R ) \abs{z_1} \alpha \norm{d\theta}_1 -
	\frac{2\KK}{4\delta} \norm{d\theta}_1 \\
&\geq
	-\frac{2\KK}{3\delta} \norm{d\theta}_1.
\end{equs}
where in the second inequality we have used that $z_2(d\theta)^2 - z_2^2>0$.
For any $d\theta \in \set{\abs{z_0 d\theta_0 + F(z_0) d\theta_1} \leq \alpha \norm{d\theta}_1} \cap \set{\abs{\bar z_0 d\theta_0 + F(\bar z_0) d\theta_1} \geq \alpha \norm{d\theta}_1}$, we can then compute that
\begin{equs}
\int_{\Rr}&\pc{ C[x](\theta + d\theta) - C[x](\theta) } \mu_{0,\tiny\text{pp}}(d x)
\\&=
	\frac{1-\epsilon}{2}\pc{
		\int_{\Rr} C[x](\theta + d\theta) - C[x](\theta) \delta(z_0) \d x	+
		\int_{\Rr} C[x](\theta + d\theta) - C[x](\theta) \delta(\bar z_0) \d x
	} \\
&\geq
	\frac{1-\epsilon}{2} \pc{
		\frac{3\KK}{4\delta} \norm{ d\theta }_1
	 {-\frac{2\KK}{3\delta} \norm{d\theta}_1}
	} \\
&\geq
	\frac{1-\epsilon}{2} \frac{\KK}{12\delta} \norm{d\theta}_1
\end{equs}
where $\mu_{0,\text{pp}}$ denotes the discrete part of the measure $\mu_0$. 
\end{proof}

\begin{Lemma}\label{lem:cones dont overlap}
For any $x_0 \neq y_0$ such that $x_0,y_0 < 0$ and $F(x_0),F(y_0) > 0$, there exists an $\alpha > 0$ such that
\begin{equ}
\set{\abs{x_0 d\theta_0 + F(x_0) d\theta_1} \leq \alpha \norm{d\theta}_1} \cap \set{\abs{y_0 d\theta_0 + F(y_0) d\theta_1} \leq \alpha \norm{d\theta}_1}
=
\set{(0,0)}.
\end{equ}
\end{Lemma}

The above directly implies that
\begin{Cor}\label{cor:cones}
  For the same $\delta, \KK>0$ as determined by Lemma \ref{lem:bound cost diff outside cone}, and for every $d \theta \in \mathbb R^d$ with $\|d \theta \| < \delta$  we have that
  \begin{equ}
  \int_{\Rr} \pc{C[x](\theta + d\theta) - C[x](\theta) }\mu_{0,\tiny\text{pp}}(\d x)
  \geq
  \frac{1-\epsilon}{2} \frac{\KK}{12\delta} \norm{d\theta}_1\,.
  \end{equ}
   Here $\mu_{0,\tiny\text{pp}}$ denotes the pure-point, or discrete, part of $\mu_0$.
\end{Cor}

\begin{proof}[Proof of Corollary~\ref{cor:cones}]
  For our choice of parameters \eqref{eq:cont example parameters}, Lemma \ref{lem:cones dont overlap} gives us an $\alpha>0$ such that the cones determined by $x_0$ and $y_0$ in parameter space do \emph{not} overlap except at $d\theta = (0,0)$. By Lemma~\ref{lem:bound cost diff outside cone}, for any such $\alpha$ there exists a $\delta>0$ such that the claim holds.
\end{proof}

\begin{proof}[Proof of Lemma \ref{lem:cones dont overlap}]
By assumption, the lines
$$
\frac{d\theta_1}{d\theta_0} = \frac{-x_0}{F(x_0)}
\quad \text{and} \quad
\frac{d\theta_1}{d\theta_0} = \frac{-y_0}{F(y_0)}
$$
only intersect at the origin. Then, the bisector of the central angle between these two lines can be written as
\begin{equ}
\frac{d\theta_1}{d\theta_0}
=
\frac{
	-\frac{\sqrt{x_0^2 + F(x_0)^2}}{\sqrt{y_0^2 + F(y_0)^2}}y_0 - x_0
}{
	\frac{\sqrt{x_0^2 + F(x_0)^2}}{\sqrt{y_0^2 + F(y_0)^2}}F(y_0) + F(x_0)
}
\eqqcolon
M > 0\,.
\end{equ}
We can assume without loss of generality that the line corresponding to $x_0$ lies \emph{above} the bisector (\ie $-x_0/F(x_0) > M$). Then, it suffices to choose $\alpha$ such that neither of the sets
\begin{equ}\label{e:sets}
  \set{\abs{x_0 d\theta_0 + F(x_0) d\theta_1} \leq \alpha \norm{d\theta}_1} \qquad, \qquad \set{\abs{y_0 d\theta_0 + F(y_0) d\theta_1} \leq \alpha \norm{d\theta}_1}
  \end{equ}
  intersects the line $\ell_M = \set{(d\theta_0, d\theta_1)\in \mathbb R^2~:~d\theta_0= M d\theta_1}$ outside of $(0,0)$. 

Notice that the boundaries of the sets in \eqref{e:sets} that are closest to $\ell_M$ can be explicitly written as the lines
\begin{gather*}
\frac{d\theta_1}{d\theta_0} = \frac{-\alpha - x_0}{\alpha + F(x_0)}
\quad \text{and} \quad
\frac{d\theta_1}{d\theta_0} = \frac{\alpha - y_0}{-\alpha + F(y_0)}\,.
\end{gather*}
Then, noting that we can replace $-x_0$ with $\abs{x_0}$ and $-y_0$ with $\abs{y_0}$ since $x_0,y_0 < 0$, we have that the only intersection of the sets in \eqref{e:sets} with $\ell_M$ is $(0,0)$ if
\begin{equ}
M < \frac{-\alpha - x_0}{\alpha + F(x_0)}
\quad \text{and} \quad
M > \frac{\alpha - y_0}{-\alpha + F(y_0)}\,,
\end{equ}
which is equivalent to the conditions
\begin{equ}
\alpha < \frac{\abs{x_0} - MF(x_0)}{1+M}
\quad \text{and} \quad
\alpha < \frac{MF(y_0) - \abs{y_0}}{1+M}\,.
\end{equ}
Since we have assumed that
\begin{equ}
\frac{\abs{x_0}}{F(x_0)} > M > \frac{\abs{y_0}}{F(y_0)}
\end{equ}
and that $F(x_0),F(y_0) > 0$, we know that the bounds in both constraints are strictly positive. Thus, we can choose any
\begin{equ}
\alpha
<
\frac{1}{2}\min\pc{\frac{\abs{x_0} - MF(x_0)}{1+M}, \frac{MF(y_0) - \abs{y_0}}{1+M}}
\end{equ}
and this proves the claim.
\end{proof}

\begin{Lemma}\label{lem:bound contribution of uniform measure}
There exists $\KK'>0$ such that for any initial point $u_0 \sim U(-5,5)$, the average cost difference can be bounded from above by
\begin{equ}
\abs*{\int_{-5}^{5} (C[u](\theta + d\theta) - C[u](\theta)) \frac{\d u}{10} }
\leq
\KK' \norm{d\theta}_1\,.
\end{equ}
\end{Lemma}

\begin{proof}[Proof of Lemma \ref{lem:bound contribution of uniform measure}]
For any $u_0 \in [-5,5]$, we can write the cost difference as
\begin{equs}
\Delta C[u_0](d\theta)
&\coloneqq
	C[u_0](\theta + d\theta) - C[u_0](\theta) \\
&=
	\pc{u_0^2 + R u_1(d\theta)^2 + \gamma C[u_1(d\theta)](\theta + d\theta)} -
	\pc{u_0^2 + R(u_1)^2 + \gamma C[u_1](\theta)} \\
&=
	\gamma \Delta C[u_1(d\theta)](d\theta) +
	R(u_1(d\theta)^2 - (u_1)^2) +
	\gamma \pc{ C[u_1(d\theta)](\theta) - C[u_1](\theta)} \\
&\leq
	\gamma\Delta C[u_0](d\theta) +
	R(u_1(d\theta)^2 - (u_1)^2) +
	\gamma \pc{ C[u_1(d\theta)](\theta) - C[u_1](\theta) } \\
&\leq
  	\gamma\Delta C[u_0](d\theta) +
  	2Ru_1\pd{F(u_0) d\theta_1 + u_0 d\theta_0} +
  	\gamma \pc{ C[u_1(d\theta)](\theta) - C[u_1](\theta)}\\
&\leq
    	\gamma\Delta C[u_0](d\theta) +
    	2R(2.5)(5) \norm{d\theta}_1 +
    	\gamma \pc{ C[u_1(d\theta)](\theta) - C[u_1](\theta)},
\end{equs}
since we assume, in the construction of the example, that $0 < R < \gamma$.
Iterating the bound established above we obtain
\begin{align}
\notag \Delta C[u_0](d \theta)
&\leq
	\sum_{t=0}^\infty \gamma^t  \pc{ 2R(2.5)(5) \norm{d\theta}_1 +
	\gamma \pc{ C[u_1(d\theta)](\theta) - C[u_1](\theta)}}\\
&=
	\frac {2R(2.5)(5) \norm{d\theta}_1 +
	\gamma \pc{ C[u_1(d\theta)](\theta) - C[u_1](\theta)}}{1-\gamma}\label{e:DeltaC}
\end{align}

We now proceed to estimate the numerator of the above expression. To do so, let $L = \frac{\abs{\omega_0} + \abs{\omega_1}}{\delta_1}$ denote the Lipschitz constant of $F$ on $(-5,5)$. Notice that in the interval $(-1, 1)$, the Lipschitz constant of $F$ restricted to this interval is $\abs{\omega_0}$ and that on $(-5,5)$ $F$ is non-expansive ($\abs{F(x)} \leq \abs{x}$) and $F^2$ is contractive ($\exists k \in [0,1)$ s.t. $\abs{F^2(x)} \leq k\abs{x}$) by construction. Since $5\abs{\omega_0}^2 > 1$ but $5\abs{\omega_0}^3 < 1$, any initial point will take at most three steps before it is mapped into the region with Lipschitz constant $\abs{\omega_0}$.
\begin{equs}
C[u_1(d\theta)](\theta) & - C[u_1](\theta)
\leq
	(1+R)\sum_{t=0}^{\gamma} \gamma^t \bigg(
		\underbrace{F \circ \cdots \circ F}_{\text{$t$ times}} (u_1(d\theta))^2 -
		\underbrace{F \circ \cdots \circ F}_{\text{$t$ times}} (u_1)^2
	\bigg) \\
&\leq
	(1+R) \pc{
		\sum_{t=0}^{3} \gamma^t L^{2t} \pc{u_1(d\theta)^2 - (u_1)^2} +
		\gamma^4 \sum_{t=0}^{\infty} \gamma^t \abs{\omega_0}^{2t}
		\pc{u_1(d\theta)^2 - (u_1)^2}
	} \\
&\leq
	(1+R)\pc{ L^6 + \frac{\gamma^4}{1-\gamma \omega_0^2} }\pc{u_1(d\theta)^2 - (u_1)^2} \\
&=
	(1+R)\pc{ L^6 + \frac{\gamma^4}{1-\gamma \omega_0^2} }
	\pc{
		2u_1\pc{F(u_0) d \theta_1 + u_0 d\theta_0} +
		\pc{F(u_0) d \theta_1 + u_0 d\theta_0}^2
	} \\
&\leq
	L'' \norm{d\theta}_1,
\end{equs}
for $\norm{d \theta}$ small enough, taking only the first order term. Combining the above with \eqref{e:DeltaC} and defining $L' = \gamma L'' + 25 R$ finally gives
\begin{equ}\label{e:k'}
\Delta C[u_0](d\theta)
\leq
\KK' \norm{d\theta}_1\,, \qquad \text{for }\qquad \KK' := \frac{L'}{1-\gamma} \norm{d\theta}_1\,,
\end{equ}
as claimed.
\end{proof}

\CounterexampleProp*

\begin{proof}[Proof of Proposition \ref{prop:counterex local min}]
Recalling the definitions of $\KK,\KK'$ from \eqref{e:k} and \eqref{e:k'} respectively, for any
\begin{align}\label{e:epsilon}
0<\epsilon < \frac{\KK}{24\delta \KK' + \KK}\,,
\end{align}
and for $d\theta \neq 0$ sufficiently small, since $\mu_0 = \mu_{0,\tiny\text{pp}} + \frac{\epsilon}{10}$ for $\mu_{0,\tiny\text{pp}} = \frac{1-\epsilon}{2}\delta_{x_0} + \frac{1-\epsilon}{2}\delta_{y_0}$ we have that
\begin{equs}
\int_{-5}^{5} (C(\theta + d\theta) - &C(\theta)) \d \mu_0(x)
=\\
&=
	\frac{1-\epsilon}{2}
	\int_{-5}^{5} \pc{C(\theta + d\theta) - C(\theta)} \d \mu_{0,\tiny\text{pp}} +
	\epsilon \int_{-5}^{5} \pc{C(\theta + d\theta) - C(\theta)} \frac{\d x}{10} \\
&\geq
	\frac{1-\epsilon}{2} \frac{\KK}{12\delta} \norm{ d \theta }_1 -
	\epsilon \KK' \norm{d\theta}_1 \\
&=
	\frac{\KK}{24\delta}\norm{d\theta}_1 -
	\epsilon\pc{\KK' + \frac{\KK}{24\delta}} \norm{d\theta}_1
	>
	0,
\end{equs}
where in the second line we have combined Corollary~\ref{cor:cones} and Lemma \ref{lem:bound contribution of uniform measure}.
Thus, we see that the point in parameter space corresponding to $\theta = (0,1)$ is a local minimum of the cost function.

\end{proof}

\section{Proofs of Convergence of the Homotopy Algorithm} \label{apdx:homotopy proofs}

For notational convenience, we define
\begin{equs}
& \AAA \coloneqq A+BK \\
& \ovl{Q} \coloneqq Q + K\transpose RK
\end{equs}
We use $\dtheta = (\dtheta_1, \dots \dtheta_d)$ to denote a perturbation of the parameters. We study the behavior of the parameters in the setting where $\theta = \theta^* + \dtheta$.

\begin{Lemma}\label{lem:continuity of P}
For a fixed quadruple $(A,B,Q,R)$ that is controllable and observable, let $P_\gamma$ denote the unique positive semi-definite solution to the discounted discrete algebraic Riccati equation (discounted DARE) with discount factor $\gamma \in [0,1]$
\begin{equ}
P_\gamma
=
\gamma A^\intercal P_\gamma A - \gamma^2 A^\intercal P_\gamma B
\pc{
	R + \gamma B^\intercal P_\gamma B
}^{-1}
B^\intercal P_\gamma A + Q
\end{equ}
Then, $f: (0,1) \to \Rr^{n \times n}$, $f(\gamma) = P_\gamma$ is a continuous function.
\end{Lemma}

\begin{proof}[Proof of \lref{lem:continuity of P}]
For any $\gamma \in [0,1]$, since $(A,B)$ is controllable, we know that $(\sqrt{\gamma} A, B)$ also controllable and thus stabilizable. Similarly, $(\sqrt{\gamma} A, D)$ is also observable, and by the positive definiteness of $R$ and positive semidefiniteness of $P_\gamma$, $(R + \gamma B^\intercal P_\gamma B)^{-1}$ exists and is continuous. By \cite[Theorem 2.4]{continuityriccati}, we conclude that $P_\gamma$ is a continuous function of $\gamma \in [0,1]$.
\end{proof}

Since the optimal policy is a continuous function of $P_\gamma$, \lref{lem:continuity of P} immediately gives us
\begin{Cor}\label{cor:continuity of K}
The optimal policy $K_\gamma^*$ is continuous in $\gamma$.
\end{Cor}

We preface the next two lemmas by reminding the reader that the expansion of the cost function that we consider is around the \emph{optimal} policy $K_{\gamma}^*$. However, since this notation is a bit tedious, we simply use $K$ to refer to this optimal policy in our coming computations. Here, we are using Assumption \ref{asm:linear functions in span} when we stipulate that $\pi_{\theta^*} = K_{\gamma}^*$.

\begin{Lemma} \label{lem:firstorderapprox}
The first-order term in the Taylor expansion of the cost (with respect to the parameters) evaluates to zero at optimality.
\end{Lemma}

Let $\pi$ denote the policy and let $\tilde{\pi}$ denote the perturbation of $\pi$ (\ie $\pi_{\theta^*+\dtheta} = \pi_{\theta^*} + \tilde{\pi}$). Furthermore, recall that our policy representation is linear in the parameters. We first compute the zero'th, first, and second order terms of the expansion of the trajectory.
\begin{equs}
x_{t}^{(0)}
&=
	\AAA^t x_0 \\
x_{t}^{(1)}
&=
	\sum_{l=0}^{t-1} \AAA^l B \tilde{\pi}( x_{t-l-1}^{(0)} ) \\
x_{t}^{(2)}
&=
	\sum_{l=1}^{t-1} \AAA^l B \pc{
		\tilde{\pi}( x_{t-l-1}^{(0)} + x_{t-l-1}^{(1)} ) -
		\tilde{\pi}( x_{t-l-1}^{(0)} )
	}
\end{equs}
where we ignore all non-second-order terms to reach the last equality. We note that, by the Lipschitz continuity of policy $\pi$, all remaining terms in the first order term of the Volterra expansion for the given system are of order higher than 1.

We can bound the Euclidean norm of the second order term in the expansion using the assumed Lipschitz continuity of $\pi$:
\begin{equs}
\norm{x_{t}^{(2)}}
&=
	\norm*{ \sum_{l=1}^{t-1} \AAA^l B \pc{
		\tilde{\pi}( x_{t-l-1}^{(0)} + x_{t-l-1}^{(1)} ) -
		\tilde{\pi}( x_{t-l-1}^{(0)} )
	} } \\
&\leq
	\sum_{l=1}^{t-1} \norm{ \AAA^l B } \norm{
		\tilde{\pi}( x_{t-l-1}^{(0)} + x_{t-l-1}^{(1)} ) -
		\tilde{\pi}( x_{t-l-1}^{(0)} )
	} \\
&\leq
	\text{Lip}(\tilde{\pi}) \sum_{l=1}^{t-1} \norm{\AAA^l B} \norm{ x_{t-l-1}^{(1)} }
\end{equs}

\begin{proof}[Proof of \lref{lem:firstorderapprox}]
We consider a perturbation of the parameters of the policy. For any $x_0 \in \supp(\rho_0)$, the first order term of the Taylor approximation is
\begin{equs}
C^{(1)}[x_0]
&=
	\sum_{t=1}^{\infty} \gamma^t
	\pc{ \pc{ x_t^{(1)} }\transpose \ovl{Q} x_t^{(0)} +
	\pc{ x_t^{(0)} }\transpose \ovl{Q} \pc{ x_t^{(1)} } }
	\\ &\qquad +
	\sum_{t=0}^{\infty} \gamma^t \pc{
		\tilde{\pi}( x_{t}^{(0)} )\transpose
		RK x_t^{(0)} +
		\pc{ x_t^{(0)} }\transpose K\transpose R
		\tilde{\pi}( x_{t}^{(0)} )
	} \\
&=
	\sum_{t=1}^{\infty} \gamma^t \sum_{l=0}^{t-1}
	\pc{\AAA^{t-l-1} B\tilde{\pi}( x_{l}^{(0)} )}\transpose \ovl{Q} \AAA^t x_0
	+
	\sum_{t=1}^{\infty} \gamma^t
	\pc{ \AAA^t x_0 }\transpose \ovl{Q} \sum_{l=0}^{t-1}
	\pc{ \AAA^{t-l-1} B \tilde{\pi}( x_{l}^{(0)} ) }\\
	&\qquad +
	\sum_{t=0}^{\infty} \gamma^t \pc{
		\tilde{\pi}( x_{t}^{(0)} )\transpose RK \AAA^t x_0 +
		\pc{\AAA^t x_0}\transpose K\transpose R \tilde{\pi}( x_{t}^{(0)} )
	} \\
&=
	\sum_{t=1}^{\infty} \gamma^t \pc{ \sum_{l=0}^{t-1}
	\pc{ \AAA^{t-l-1} B \tilde{\pi}( x_{l}^{(0)} ) }\transpose \ovl{Q}\AAA^t x_0
	+
	\pc{\AAA^t x_0}\transpose \ovl{Q}
	\sum_{l=0}^{t-1} \omega  \pc{ \AAA^{t-l-1} B \tilde{\pi}( x_{l}^{(0)} ) } }\\
	&\qquad +
	\sum_{t=0}^{\infty} \gamma^t \pc{
	\tilde{\pi}( x_{t}^{(0)} )\transpose RK \AAA^t x_0 +
	\pc{\AAA^t x_0}\transpose K\transpose R \tilde{\pi}( x_{t}^{(0)} ) }\\
\intertext{Let $s=t-l$, changing the order of summation, we get}
&=
	\sum_{l=0}^{\infty} \gamma^l
	\tilde{\pi}( x_{l}^{(0)} )\transpose \pq{
		B\transpose \pc{
			\sum_{s=1}^{\infty}\gamma^s \pc{\AAA^{s-1}}\transpose \ovl{Q} \AAA^{s-1}
		} \AAA
	} \AAA^l x_0
	\\ &\qquad +
	\sum_{t=0}^{\infty} \gamma^t
	\tilde{\pi}( x_{t}^{(0)} )\transpose\pq{RK}\AAA^t x_0
	\\ &\qquad +
	\sum_{l=0}^{\infty} \gamma^l
	\pc{\AAA^l x_0}\transpose \pq{
		\AAA\transpose \pc{
			\sum_{s=1}^{\infty}\gamma^s \pc{\AAA^{s-1}}\transpose \ovl{Q} \AAA^{s-1}
		} B
	} \tilde{\pi}( x_{l}^{(0)} )
	\\ &\qquad +
	\sum_{t=0}^{\infty} \gamma^t
	\pc{\AAA^t x_0}\transpose\pq{K\transpose R} \tilde{\pi}( x_{t}^{(0)} ) \\
&=
	\sum_{l=0}^{\infty} \gamma^l \tilde{\pi}( x_{l}^{(0)} )\transpose
	\pc{ \gamma B\transpose P_\gamma \pc{ A + BK } + RK } \AAA^l x_0
	\\ &\qquad +
	\sum_{l=0}^{\infty} \gamma^l \pc{\AAA^l x_l}\transpose
	\pc{ \gamma \pc{ A + BK }\transpose P_\gamma B + K\transpose R }
	\tilde{\pi}( x_{l}^{(0)} )
\end{equs}
which evaluates to zero at $K^*$ since $K^* = -\gamma \pc{R + \gamma B\transpose P_\gamma B}^{-1} B\transpose P_\gamma A$. This is consistent with the fact that the first order term is the gradient multiplied by $\dtheta$. This calculation gives us an explicit form of the term that will simplify the proof of the next lemma.
\end{proof}

\begin{Lemma} \label{lem:secondorderapprox}
Let $\rho_0$ have full support, then the Hessian is symmetric and strictly positive definite at optimality.
\end{Lemma}

\begin{proof}[Proof of \lref{lem:secondorderapprox}]
For any $x_0 \in \supp(\rho_0)$, the second order term in the approximation of the cost is as follows.
\begin{equs}
C^{(2)}[x_0]
&=
	\sum_{t=2}^{\infty} \gamma^t
	\pc{
		x_t^{(2)}
	}\transpose
	\ovl{Q} x_t^{(0)}
	+
	\sum_{t=2}^{\infty} \gamma^t
	\pc{
		x_t^{(0)}
	}\transpose
	\ovl{Q} x_t^{(2)}
	\\ &\qquad +
	\sum_{t=1}^{\infty} \gamma^t
	\pc{
		x_t^{(1)}
	}\transpose
	\ovl{Q} x_t^{(1)} +
	\sum_{t=0}^{\infty} \mathunderline{red}{
		\gamma^t
		\pc{
			\tilde{\pi}(x_t^{(0)})
		}\transpose
		R
		\pc{
			\tilde{\pi}(x_t^{(0)})
		}
	} \\
	&\qquad +
	\sum_{t=1}^{\infty} \gamma^t
	\pc{
		\pc{ x_t^{(1)} }\transpose K\transpose R \tilde{\pi}(x_t^{(0)}) +
		\pc{
			\tilde{\pi}( x_t^{(0)} + x_t^{(1)} ) -
			\tilde{\pi}( x_t^{(0)} )
		}\transpose
		RK x_t^{(0)}
	} \\
	&\qquad +
	\sum_{t=1}^{\infty} \gamma^t
	\pc{
		\tilde{\pi}(x_t^{(0)})\transpose R Kx_t^{(1)}  +
		\pc{x_t^{(0)}}\transpose K\transpose
		R\pc{
			\tilde{\pi}( x_t^{(0)} + x_t^{(1)} ) -
			\tilde{\pi}( x_t^{(0)} )
		}
	} \\
\intertext{Notice we can ignore the underlined sum since it is already symmetric. In the remainder of this calculation, we will use a red underline to highlight the symmetric terms that we will omit for clarity. The remaining terms of interest are}
&=
	\sum_{t=1}^{\infty} \gamma^t
	\pc{x_t^{(1)}}\transpose
	\ovl{Q}
	x_t^{(1)}
	+
	\sum_{t=2}^{\infty} \gamma^t \sum_{l=1}^{t-1}
	\pc{
		\tilde{\pi}( x_l^{(0)} + x_l^{(1)} ) -
		\tilde{\pi}( x_l^{(0)} )
	}\transpose
	B\transpose \pc{\AAA^{t-l-1}}\transpose
	\ovl{Q} \AAA^t x_0
	\\
	&\qquad +
	\sum_{t=2}^{\infty} \gamma^t \sum_{l=1}^{t-1}
	\pc{
		\AAA^t x_0
	}\transpose
	\ovl{Q} \AAA^{t-l-1} B \pc{
		\tilde{\pi}( x_l^{(0)} + x_l^{(1)} ) -
		\tilde{\pi}( x_l^{(0)} )
	}
	\\ &\qquad +
	\sum_{t=1}^{\infty} \gamma^t \pc{
		\pc{x_t^{(1)}}\transpose
		K\transpose R \tilde{\pi}( x_t^{(0)} ) +
		\pc{
			\tilde{\pi}( x_t^{(0)} + x_t^{(1)} ) -
			\tilde{\pi}( x_t^{(0)} )
		}\transpose RK \AAA^t x_0
	}
	\\ &\qquad +
	\sum_{t=1}^{\infty} \gamma^t \pc{
		\tilde{\pi}( x_t^{(0)} )\transpose
		R Kx_t^{(1)}  +
		\pc{ \AAA^t x_0}\transpose K\transpose
		R \pc{
			\tilde{\pi}( x_t^{(0)} + x_t^{(1)} ) -
			\tilde{\pi}( x_t^{(0)} )
		}
	} \\
\intertext{Like the proof of \lref{lem:firstorderapprox}, we define $s=t-l$ and change the order of summation to get}
&=
	\sum_{l=1}^{\infty} \gamma^l \pc{
		\tilde{\pi}( x_l^{(0)} + x_l^{(1)} ) -
		\tilde{\pi}( x_l^{(0)} )
	}\transpose
	\textcolor{orange}{\Bigg[}
		B\transpose
		\sum_{s=1}^{\infty} \gamma^s
		\pc{\AAA^{s-1}}\transpose \ovl{Q} \AAA^{s-1} \pc{A + BK}
	\textcolor{orange}{\Bigg]} x_l^{(0)}
	\\	&\qquad +
	\sum_{l=1}^{\infty} \gamma^l \pc{x_l^{(0)}}\transpose
	\textcolor{orange}{\Bigg[}
		(A+BK)\transpose
		\sum_{s=1}^{\infty} \gamma^s
		\pc{\AAA^{s-1}}\transpose \ovl{Q} \AAA^{s-1} B
	\textcolor{orange}{\Bigg]} \pc{
		\tilde{\pi}( x_l^{(0)} + x_l^{(1)} ) -
		\tilde{\pi}( x_l^{(0)} )
	}
	\\ &\qquad +
	\sum_{t=1}^{\infty} \gamma^t \pc{
		\tilde{\pi}( x_t^{(0)} + x_t^{(1)} ) -
		\tilde{\pi}( x_t^{(0)} )
	}\transpose
	\textcolor{orange}{\Bigg[}
		RK
	\textcolor{orange}{\Bigg]} x_t^{(0)}
	+
	\sum_{t=1}^{\infty} \gamma^t \pc{ x_t^{(1)} }\transpose
	K\transpose R \tilde{\pi}( x_t^{(0)} )
	\\ &\qquad +
	\sum_{t=1}^{\infty} \gamma^t \pc{x_t^{(0)}}\transpose
	\textcolor{orange}{\Bigg[}
		RK
	\textcolor{orange}{\Bigg]} \pc{
		\tilde{\pi}( x_t^{(0)} + x_t^{(1)} ) -
		\tilde{\pi}( x_t^{(0)} )
	}
	+
	\sum_{t=1}^{\infty} \gamma^t \tilde{\pi}( x_t^{(0)} )\transpose
	R Kx_t^{(1)}
	\\ &\qquad +
	\sum_{t=2}^{\infty} \gamma^t
	\pc{x_t^{(1)}}\transpose
	\ovl{Q} \sum_{m=1}^{t-1}
	\pc{
		\AAA^{t-m-1} B \tilde{\pi}( x_m^{(0)} )
	} +
	\sum_{t=2}^{\infty} \gamma^t \sum_{m=1}^{t-1}
	\pc{
		\AAA^{t-m-1} B \tilde{\pi}( x_m^{(0)} )
	}\transpose \ovl{Q} x_t^{(1)}
	\\ &\qquad +
	\sum_{t=1}^{\infty} \mathunderline{red}{\gamma^t
	\pc{
		\tilde{\pi}( x_0 )
	}\transpose
	B\transpose \pc{\AAA^{t-1}}\transpose
	\ovl{Q}
	\AAA^{t-1} B \tilde{\pi}( x_0 ) } \\
\intertext{We have dealt with the term containing two copies of $x_t^{(1)}$ by splitting it into two pieces: the sum of all terms where both indices of the inner sums are $0$ and everything else. The former is symmetric, so we can omit it. {Noticing that $P_\gamma = \sum_{t=0}^\infty \gamma^t (\mathcal A^t)^\top \bar Q \mathcal A^t$ we continue from above}}
&=
	\sum_{l=1}^{\infty} \gamma^l \pc{
		\tilde{\pi}( x_l^{(0)} + x_l^{(1)} ) -
		\tilde{\pi}( x_l^{(0)} )
	}\transpose
	\textcolor{orange}{\Bigg[}
		\gamma (R + \gamma B\transpose P_\gamma B)K + \gamma B\transpose P_\gamma A
	\textcolor{orange}{\Bigg]} x_l^{(0)}
	\\ &\qquad +
	\sum_{l=1}^{\infty} \gamma^l \pc{x_l^{(0)}}\transpose
	\textcolor{orange}{\Bigg[}
		\gamma K\transpose(R + \gamma B\transpose P_\gamma B) + \gamma A\transpose P_\gamma B
	\textcolor{orange}{\Bigg]} \pc{
		\tilde{\pi}( x_l^{(0)} + x_l^{(1)} ) -
		\tilde{\pi}( x_l^{(0)} )
	}
	\\ &\qquad +
	\sum_{t=2}^{\infty} \gamma^t \sum_{l=0}^{t-1}
	\tilde{\pi}( x_l^{(0)} )\transpose
	B\transpose \pc{\AAA^{t-l-1}}\transpose
	\ovl{Q}
	\sum_{m=1}^{t-1}
	\pc{
		\AAA^{t-m-1} B \tilde{\pi}( x_m^{(0)} )
	}
	\\ &\qquad +
	\sum_{t=2}^{\infty} \gamma^t \sum_{m=1}^{t-1}
	\pc{
		\AAA^{t-m-1} B \tilde{\pi}( x_m^{(0)} )
	}\transpose
	\ovl{Q}
	\sum_{l=0}^{t-1} \AAA^{t-l-1} B \tilde{\pi}( x_l^{(0)} )
	\\ &\qquad +
	\sum_{t=1}^{\infty} \gamma^t \pc{ x_t^{(1)} }\transpose
	K\transpose R \tilde{\pi}( x_t^{(0)} )
	+
	\sum_{t=1}^{\infty} \gamma^t \tilde{\pi}( x_t^{(0)} )\transpose
	R Kx_t^{(1)} \\
&=
	\sum_{l=1}^{\infty} \gamma^l \pc{
		\tilde{\pi}( x_l^{(0)} + x_l^{(1)} ) -
		\tilde{\pi}( x_l^{(0)} )
	}\transpose
	\textcolor{orange}{\Bigg[}
		\gamma (R + \gamma B\transpose P_\gamma B)K + \gamma B\transpose P_\gamma A
	\textcolor{orange}{\Bigg]} x_l^{(0)}
	\\ &\qquad +
	\sum_{l=1}^{\infty} \gamma^l \pc{x_l^{(0)}}\transpose
	\textcolor{orange}{\Bigg[}
		\gamma K\transpose (R + \gamma B\transpose P_\gamma B) + \gamma A\transpose P_\gamma B
	\textcolor{orange}{\Bigg]} \pc{
		\tilde{\pi}( x_l^{(0)} + x_l^{(1)} ) -
		\tilde{\pi}( x_l^{(0)} )
	}
	\\ &\qquad +
	\sum_{t=2}^{\infty} \gamma^t \sum_{l=0}^{t-1}
	\tilde{\pi}( x_l^{(0)} )\transpose
	B\transpose \pc{\AAA^{t-l-1}}\transpose
	\ovl{Q}
	\sum_{m=1}^{t-1}
	\pc{
		\AAA^{t-m-1} B \tilde{\pi}( x_m^{(0)} )
	}
	\\ &\qquad +
	\sum_{t=2}^{\infty} \gamma^t \sum_{m=1}^{t-1}
	\pc{
		\AAA^{t-m-1} B \tilde{\pi}( x_m^{(0)} )
	}\transpose
	\ovl{Q}
	\sum_{l=0}^{t-1} \AAA^{t-l-1} B \tilde{\pi}( x_l^{(0)} )
	\\ &\qquad +
	\sum_{t=1}^{\infty} \gamma^t \sum_{l=0}^{t-1}
	\pc{
		\tilde{\pi}( x_l^{(0)} )
	}\transpose
	B\transpose \pc{\AAA^{t-l-1}}\transpose
	K\transpose R \tilde{\pi}( x_t^{(0)} )
	\\ &\qquad +
	\sum_{t=1}^{\infty} \gamma^t \sum_{l=0}^{t-1}
	\tilde{\pi}( x_t^{(0)} )\transpose
	R K \AAA^{t-l-1} B \tilde{\pi}( x_l^{(0)} ) \\
&=
	\sum_{l=1}^{\infty} \gamma^l \pc{
		\tilde{\pi}( x_l^{(0)} + x_l^{(1)} ) -
		\tilde{\pi}( x_l^{(0)} )
	}\transpose
	\textcolor{orange}{\Bigg[}
		\gamma (R + \gamma B\transpose P_\gamma B)K + \gamma B\transpose P_\gamma A
	\textcolor{orange}{\Bigg]} x_l^{(0)}
	\\ &\qquad +
	\sum_{l=1}^{\infty} \gamma^l \pc{x_l^{(0)}}\transpose
	\textcolor{orange}{\Bigg[}
		\gamma K\transpose (R + \gamma B\transpose P_\gamma B) + \gamma A\transpose P_\gamma B
	\textcolor{orange}{\Bigg]} \pc{
		\tilde{\pi}( x_l^{(0)} + x_l^{(1)} ) -
		\tilde{\pi}( x_l^{(0)} )
	}
	\\ &\qquad +
	\sum_{l=0}^{\infty} \gamma^l
	\pc{
		\tilde{\pi}( x_l^{(0)} )
	}\transpose B\transpose
	\textcolor{blue}{\Bigg[}
		\sum_{s=1}^{\infty} \gamma^s
		\pc{\AAA^{s-1}}\transpose
		K\transpose R \tilde{\pi}( x_{l+s}^{(0)} )
		\\ &\qquad \qquad \qquad \qquad \qquad \qquad \qquad +
		\sum_{s=2}^{\infty} \gamma^s
		\pc{\AAA^{s-1}}\transpose
		\ovl{Q}
		\sum_{r=1}^{s-1}
		\pc{
			\AAA^{s-r-1} B \tilde{\pi}( x_{l+r}^{(0)} )
		}
	\textcolor{blue}{\Bigg]}
	\\ &\qquad +
	\sum_{l=0}^{\infty} \gamma^l
	\textcolor{blue}{\Bigg[}
		\sum_{s=1}^{\infty} \gamma^s
		\tilde{\pi}( x_{l+s}^{(0)} )\transpose
		R K \AAA^{s-1}
		\\ &\qquad \qquad \qquad \qquad +
		\sum_{s=2}^{\infty} \gamma^s
		\sum_{r=1}^{s-1}
		\pc{
			\AAA^{s-r-1} B \tilde{\pi}( x_{l+r}^{(0)} )
		}\transpose
		\ovl{Q}
		\AAA^{s-1}
	\textcolor{blue}{\Bigg]} B \tilde{\pi}( x_l^{(0)} )
\end{equs}
To finish the proof, notice that the bracketed terms are almost identical the first-order term from \lref{lem:firstorderapprox}. We also note that the terms multiplied to the bracketed terms are guaranteed to be finite due to the assumed Lipschitz continuity of $F$.
\begin{equs}
\textcolor{orange}{\Bigg[} \cdots \textcolor{orange}{\Bigg]}
&=
	0 \qquad
	\text{
		since $K^* = -\gamma
		\pc{R + \gamma B\transpose P_\gamma B}^{-1} B\transpose P_\gamma A$
	}
\\
\textcolor{blue}{\Bigg[} \cdots \textcolor{blue}{\Bigg]}
&=
	\sum_{s=2}^{\infty} \gamma^{s}
	\sum_{r=0}^{s-2}
	\pc{\AAA^{r+1} \AAA^{s-r-2}}\transpose
	\ovl{Q}
	\AAA^{s-r-2} B \pc{
		\tilde{\pi}( x_{l+r+1}^{(0)} )
	}
	\\ &\qquad +
	\sum_{s=1}^{\infty} \gamma^s \pc{\AAA^{s-1}}\transpose
	K\transpose R
	\pc{
		\tilde{\pi}( x_{l+s}^{(0)} )
	} \\
&=
	\sum_{r=0}^{\infty} \gamma^{r+1} \pc{\AAA^r}\transpose
	\Bigg[ \AAA\transpose
		\sum_{m=2}^{\infty} \gamma^{m-1} \pc{\AAA^{m-2}}\transpose
		\ovl{Q}
		\AAA^{m-2} B \Bigg]
		\pc{
			\tilde{\pi}( x_{l+r+1}^{(0)} )
		}
	\\ &\qquad +
	\sum_{s=0}^{\infty} \gamma^{s+1} \pc{\AAA^s}\transpose
	K\transpose R
	\pc{
		\tilde{\pi}( x_{l+s+1}^{(0)} )
	} \\
&=
	\sum_{r=0}^{\infty} \gamma^r \pc{\AAA^r}\transpose
	\Bigg[
		\gamma \pc{A+BK}\transpose P_\gamma B
		+ K\transpose R
	\Bigg]
	\pc{
		\tilde{\pi}( x_{l+r+1}^{(0)} )
	} \\
&=
	0
\end{equs}
Thus, the second-order term of the expansion of the cost function around optimality for the trajectory with initial condition $x_0$ reads
\begin{equs}
C^{(2)}[x_0]
&=
	\sum_{t=0}^{\infty} \gamma^t \pc{
		\tilde{\pi}( x_{t}^{(0)} )\transpose R \tilde{\pi}( x_{t}^{(0)} ) +
		\tilde{\pi}( x_{0} )\transpose B\transpose \pc{\AAA^{t}}\transpose
		\ovl{Q}	\AAA^{t} B \tilde{\pi}( x_{0} )
	} \\
&=
	\tilde{\pi}( x_{0} )\transpose B\transpose P_\gamma B \tilde{\pi}( x_{0} ) +
	\sum_{t=0}^{\infty} \gamma^t
	\tilde{\pi}( x_{t}^{(0)} )\transpose R \tilde{\pi}( x_{t}^{(0)} )
\end{equs}
Since the policy is linear in the parameters, we can formally write this as an inner product where we define multiplication between an $m \times m$ matrix and a vector of $m \times 1$ vectors to be carried out component-wise.
\begin{equs}
~&=
	\dtheta\transpose
	\pc{
		\begin{bmatrix}
			f_1(x_0) \\
			\vdots \\
			f_d(x_0)
		\end{bmatrix}\transpose
		B\transpose P B
		\begin{bmatrix}
			f_1(x_0) \\
			\vdots \\
			f_d(x_0)
		\end{bmatrix}
		+
		\sum_{t=0}^{\infty} \gamma^t
		\begin{bmatrix}
			f_1(x_t) \\
			\vdots \\
			f_d(x_t)
		\end{bmatrix}\transpose
		R
		\begin{bmatrix}
			f_1(x_t) \\
			\vdots \\
			f_d(x_t)
		\end{bmatrix}
	}
	\dtheta \\
&=
	\dtheta\transpose
	\left(
		\begin{bmatrix}
			f_1(x_0)\transpose B\transpose P_\gamma B f_1(x_0) & \cdots & f_1(x_0)\transpose B\transpose P_\gamma B f_d(x_0) \\
			\vdots & \ddots & \vdots \\
			f_d(x_0)\transpose B\transpose P_\gamma B f_1(x_0) & \cdots & f_d(x_0)\transpose B\transpose P_\gamma B f_d(x_0)
		\end{bmatrix}
	\right.
	\\ &\qquad \qquad \qquad  +
	\left.
		\sum_{t=0}^{\infty} \gamma^t
		\begin{bmatrix}
			f_1(x_t)\transpose R f_1(x_t) & \cdots & f_1(x_t)\transpose R f_d(x_t) \\
			\vdots & \ddots & \vdots \\
			f_d(x_t)\transpose R f_1(x_t) & \cdots & f_d(x_t)\transpose R f_d(x_t)
		\end{bmatrix}
	\right)
	\dtheta
\end{equs}
To avoid writing this large expression again, we denote the matrix in the parentheses by $H_{\gamma}$. We use this matrix to define the quadratic form used as the Lyapunov function in the proof of Theorem \ref{thm:local nonlinear convergence}. Given this expression for $C^{(2)}$, we notice that it is always non-negative and is zero if and only if the policy is optimal for the given initial condition, $x$ (\ie $\Delta u_t = 0 \, \forall t$). Consequently, integrating this term against the initial distribution $\rho_0$ will result in a nonzero outcome if and only if the trajectory of almost all initial condition is optimal. Since $\rho_0$ has full support, and by Assumption~\ref{asm:linear indep}
\begin{equ}
C^{(2)}
=
	\int_{\Rr} C^{(2)}[x] ~\d \rho_0(x) = 0
\iff
	C^{(2)}[x] = 0 \,\, \text{$\pq{\rho_0}$-a.e.}
\iff
	\tilde{\pi}( x_{t}^{(0)} ) = \vec{0} \,\, \forall t
\iff
	\Delta \theta = \vec{0}
\end{equ}
\end{proof}

\ThmConvOfFirstIter*

\begin{proof}[Proof of \tref{thm:convergence of first iteration}]
Let $\theta$ be the randomly initialized parameters. When $\gamma = 0$, the cost can be written as
\begin{equ}
C_0(\theta) = \Ex{x_0 \sim \rho_0}{x_0\transpose Qx_0 + u_0\transpose R u_0}
\end{equ}
For convenience, we will omit the subscripts. Notice that the optimal parameters are $\theta^* = \vec{0}$ which correspond to the optimal action of $u=0$. We can calculate
\begin{equs}
\frac{\d}{\d s} \pc{C_0(\theta(s)) - C_0(\theta_0^*)}
&=
	\E{\frac{\d}{\d s}\pc{u\transpose R u}} = \E{2 u\transpose R \frac{\d}{\d s}u} \\
&=
	\E{2 u\transpose R \pc{ \frac{\d}{\d s} \pi_{\theta(s)}(x) } } \\
&=
	\E{2 u\transpose R \dtp{
		\nabla_\theta \pi_\theta(x),
		\frac{\d}{\d s} \theta(s)
	} } \\
&=
	\E{2 u\transpose R \dtp{
		\nabla_\theta \pi_\theta(x) ,
		-\nabla_\theta C_0(\theta)
	} } \\
&=
	\E{2 u\transpose R \dtp{
		\nabla_\theta \pi_\theta(x),
		- \frac{\d}{\d u} C_0(\theta) \cdot \nabla_\theta \pi_\theta(x) }
	 } \\
&=
	\E{
		-4 \dtp{u\transpose R
			\nabla_\theta \pi_\theta(x), u\transpose R \nabla_\theta \pi_\theta(x)
		}
	} \\
&=
	-\Tr\pc{\nabla_\theta C_0\transpose \nabla_\theta C_0}
\end{equs}
where $\dtp{\cdot,\cdot}$ denotes the scalar product in $\mathbb R^d$.
To prove convergence, what remains is to show that the only extremum is the global minimum. In other words, we want to show
\begin{equ}
\nabla_\theta C_0(\theta) = 0
\iff
\theta = \theta_0^*
\end{equ}
\begin{itemize}
\item[$(\Leftarrow)$] This direction is straightforward. At $\theta=\theta_\gamma^*$, we see that $u = 0$ for any $x$, thus, $\nabla_C = 0$. \\

\item[$(\Rightarrow)$] Recall that for $f \geq 0$, $\int f \d \mu = 0 \iff f = 0 \text{ $\pq{\mu}$-a.e.}$ and if $f$ is continuous, then $f \equiv 0$. This tells us
\begin{equ}
  -4 \dtp{\pi_\theta(x)\transpose R
    \nabla_\theta \pi_\theta(x), \pi_\theta(x)\transpose R \nabla_\theta \pi_\theta(x)
  }
=-\|2\pi_\theta(x)\transpose R
  \nabla_\theta \pi_\theta(x)\|^2=
0
\quad
\forall x
\end{equ}
{This condition is equivalent to }
\begin{equ}
\pc{\sum_{k=0}^d \theta_k f_k(x) R f_i(x)}^2=
0 \quad \forall x\in \mathbb R^n, i \in \{1,\dots, d\}\end{equ}
and \begin{equ}{\sum_{k=0}^d \theta_k f_k(x) R f_i(x)}=
0\qquad \forall x \quad \forall x\in \mathbb R^n, i \in \{1,\dots, d\}
\end{equ}

Since we assumed that $\set{f_k}$ is linearly independent (Assumption \ref{asm:linear indep}), {by positive definiteness of $R$} we get that $\theta = \theta_\gamma^*$.
\end{itemize}
\end{proof}

For the next theorem, we recall the definition of the positive definite  Hessian $H_\gamma$ appearing in the proof of Lemma~\ref{lem:secondorderapprox}:
\begin{equ}
\dtheta\transpose H_\gamma \dtheta\transpose = \dtheta\transpose
\pc{
  \begin{bmatrix}
    f_1(x_0) \\
    \vdots \\
    f_d(x_0)
  \end{bmatrix}\transpose
  B\transpose P B
  \begin{bmatrix}
    f_1(x_0) \\
    \vdots \\
    f_d(x_0)
  \end{bmatrix}
  +
  \sum_{t=0}^{\infty} \gamma^t
  \begin{bmatrix}
    f_1(x_t) \\
    \vdots \\
    f_d(x_t)
  \end{bmatrix}\transpose
  R
  \begin{bmatrix}
    f_1(x_t) \\
    \vdots \\
    f_d(x_t)
  \end{bmatrix}
}
\dtheta \end{equ}

\ThmLocalConv*

\begin{proof}[Proof of \tref{thm:local nonlinear convergence}]
Consider the Lyapunov function
\begin{equ}
U(s)
=
\frac{1}{2} \pc{\theta(s) - \theta_{\gamma}^*}\transpose
H_{\gamma} \pc{\theta(s) - \theta_{\gamma}^*}
\end{equ}
where $H_{\gamma}$ is a symmetric positive-definite matrix defined at the end of Lemma \ref{lem:secondorderapprox}.

In this regime we can expand the cost function $C_{\gamma}(\,\cdot\,)$ around its minimum $\theta_{\gamma}^*$. By Lemmas \ref{lem:firstorderapprox} and \ref{lem:secondorderapprox}, the second-order Taylor expansion of the cost is
\begin{equs}
C_\gamma(\theta(s))
&=
	C_\gamma^{(0)}|_{\theta_{\gamma}^*} + C_\gamma^{(1)}|_{\theta_{\gamma}^*} +
	C_\gamma^{(2)}|_{\theta_{\gamma}^*} + o(\norm{\dtheta(s)}^2) \\
&=
	C_\gamma(\theta_{\gamma}^*) +
	\frac{1}{2} \dtheta(s)\transpose H_{\gamma} \dtheta(s) +
	o(\norm{\dtheta(s)}^2)
\\
C_\gamma(\theta(s)) - C_\gamma(\theta_{\gamma}^*)
&=
	U(s) + o(\norm{\dtheta(s)}^2)
\end{equs}
By chain rule,
\begin{equs}
\frac{d}{d s}U(s)
&=
	\dtp{
		\nabla_{\dtheta} U(s),
		\frac{d}{d s}\dtheta(s)
	} \\
&=
	\dtp{
		H_{\gamma}\dtheta(s),
		-\nabla_{\theta}C(\theta(s))
	} \\
&=
	-\dtheta(s)\transpose H_{\gamma}^2 \dtheta(s) +
	o(\norm{\dtheta(s)}^2) \\
&\leq
	-\lambda_{\mathrm{min}}(H_{\gamma}) U(s) +
	o(\norm{\dtheta(s)}^2)
\end{equs}
where $\lambda_{\text{min}}(H_\gamma)$ is the smallest  of $H_\gamma$. By Gr\"onwall's Inequality, this immediately gives, for $\|\dtheta(0)\|< \delta$ for $\delta>0$ sufficiently small, that
\begin{equ}
U(s) \leq e^{-s\lambda_{\mathrm{min}}(H_{\gamma})/2} U(0)\,,
\end{equ}recovering the claim for $\lambda = \lambda_{\mathrm{min}}(H_{\gamma})/2$.
\end{proof}

\end{document}